\documentclass[11pt]{amsproc}

\usepackage{mathtext}
\usepackage[cp1251]{inputenc}

\usepackage{bm}
\usepackage[dvips]{graphicx}
\usepackage{amsmath}
\usepackage{amssymb}
\usepackage{amsxtra}

\usepackage{epsfig}
\usepackage{epic}
\usepackage{eepic}
\usepackage{graphics}
\usepackage{graphicx}
\usepackage{subfigure}

\usepackage{caption}
\def\N{{{\Bbb N}}}
\def\Z{{{\Bbb Z}}}
\def\T{{{\Bbb T}}}
\def\R{{\Bbb R}}

\def\l{{\lambda }}
\def\a{{\alpha }}
\def\D{{\Delta }}

\def\a{{\alpha}}

\def\d{{\delta}}

\def\s{{\sigma}}
\def\vp{{\varphi}}

\def\g{{\gamma }}
\def\w{{\omega }}
\def\L{{\mathcal{L} }}

\def\){\right)}
\def\({\left(}
\def\supp{\operatorname{supp}}

\numberwithin{equation}{section}
\newtheorem{corollary}{Corollary}[section]
\newtheorem{lemma}{Lemma}[section]
\newtheorem{theorem}{Theorem}[section]
\newtheorem{proposition}{Proposition}[section]
\newtheorem{remark}{Remark}[section]

\newtheorem*{assuma}{Assumption A}

\theoremstyle{definition} \newtheorem{example}{Example}[section]

\par

\begin{document}

\title[]{Approximation by linear sampling operators in Banach spaces}

\author[Yurii
Kolomoitsev]{Yurii
Kolomoitsev$^{\text{a, 1}}$}
\address{Institute for Numerical and Applied Mathematics, G\"ottingen University, Lotzestr. 16-18, 37083 G\"ottingen, Germany}
\email{kolomoitsev@math.uni-goettingen.de}

\thanks{$^\text{a}$Institute for Numerical and Applied Mathematics, G\"ottingen University, Lotzestr. 16-18, 37083 G\"ottingen, Germany}

\thanks{$^1$Supported by the German Research Foundation, project KO 5804/3-1}


\thanks{E-mail address: kolomoitsev@math.uni-goettingen.de}

\date{\today}

\begin{abstract}
This paper studies approximation properties of linear sampling operators in general Banach lattices $X$. We obtain matching direct and inverse approximation estimates, convergence criteria, equivalence results involving special $K$-functionals and their realizations by sampling operators, as well as strong converse inequalities, which, to the best of our knowledge, have not been previously established for sampling operators even in the classical spaces $L_p$. The results extend several classical theorems previously known mainly in $L_p$ and apply to all functions $f\in X$ for which the corresponding sampling operator is well defined, thereby substantially enlarging the class of functions that can be considered in this framework.
\end{abstract}

\maketitle

\section{Introduction}

Classical results in approximation theory and Fourier analysis include direct
and inverse theorems for approximation by trigonometric polynomials and entire
functions of exponential type. Such results were first obtained in the space $C$
of uniformly continuous functions and in the spaces $L_p$, $1\le p<\infty$, using
polynomials (or entire functions) of best approximation  and convolution operators
(see, e.g.,~\cite{DL, Z, timan, TB}). Since then, they have been extended to a
wide range of more general function spaces and approximation methods.

The most studied and technically simplest case is the approximation by convolution operators in translation-invariant Banach spaces. In this setting, Young-type inequalities hold, and in many situations the approximation problem essentially reduces to the investigation of the behavior of the convolution kernels in the space $L_1$ (see, e.g.,~\cite[Ch.~9]{S71}, \cite[Ch.~7]{DL}, and~\cite{D80}).

In more general Banach spaces, which are typically not translation-invariant (for instance, weighted Lebesgue, Lorentz, Orlicz, and variable exponent spaces), the problem of obtaining direct and inverse approximation estimates becomes more delicate. In such spaces, the analysis often relies on tools from interpolation theory, the boundedness of the Hardy--Littlewood maximal operator, or the uniform boundedness of Ces\`aro means. Moreover, the study of approximation processes usually requires the introduction of special measures of smoothness, such as $K$-functionals, and their realizations or special moduli of smoothness; see, e.g.,~\cite{D98, IG06, KY10, V23, Vo23}.

A notable recent contribution in this direction is~\cite{V23}, where direct and inverse approximation theorems for convolution operators were obtained in Banach lattices $X$ under the uniform boundedness of the Steklov averaging operators. This provides, to the best of our knowledge, one of the most general settings currently known for such approximation results, with smoothness characterized by moduli generated by Steklov averages.

In recent years, considerable attention has also been paid to the approximation properties of Kantorovich-type sampling operators. Approximation results for such operators have been established in various Banach spaces, including weighted $L_p$ spaces~\cite{KS21}, Wiener classes~\cite{KKS20}, Orlicz spaces~\cite{CV14, CPV23}, and weighted spaces with variable exponent~\cite{D25}.

The case of sampling operators is substantially more involved.
In contrast to convolution and Kantorovich-type operators, considerably fewer approximation results are available in general Banach spaces.
Error estimates for classical Lagrange interpolation polynomials in weighted spaces $L_{p,w}(\T)$ with Muckenhoupt weights were obtained in~\cite[Theorem~3.2.15]{MM08}.
In non-periodic weighted $L_p$ spaces, approximation properties of various sampling and interpolation operators were studied in~\cite{NU19} and~\cite{KS21}.
The case of mixed-norm spaces $L_{\bar{p}}$  was considered in~\cite{CZ25}.
In~\cite{FG93}, the authors studied the approximation of bounded continuous functions by sampling operators in Banach lattices  $X$ satisfying a Young-type convolution inequality.
For results on the reconstruction of band-limited functions from irregular sampling in such Banach lattices, see~\cite{FG92}; see also~\cite{IS13} for the case of more general Banach lattices.
In the setting of translation-invariant Banach lattices $X(\T)$, approximation estimates for Lagrange interpolation polynomials in terms of moduli of smoothness and the error of best one-sided approximation were established in~\cite{KT26}.

The aim of the present paper is to study the approximation properties of linear sampling operators in general Banach lattices $X$.
Before turning to this general setting, we briefly review the relevant approximation properties of sampling operators in the classical spaces $L_p(\T)$, $1<p<\infty$, using the following trigonometric Lagrange interpolation polynomial on equidistant nodes as an example:
\begin{equation*}
  \mathcal{{L}}_n f(x)
  := \frac{1}{2n+1}\sum_{k=0}^{2n} f(t_k)
    \frac{\sin\bigl((n+\tfrac12)(x-t_k)\bigr)}{\sin\bigl(\tfrac{x-t_k}{2}\bigr)},
  \qquad
  t_k=\frac{2\pi k}{2n+1}.
\end{equation*}

Let $f\in C(\T)$ and $1<p<\infty$. By the classical Marcinkiewicz--Zygmund inequality (see~\cite[7.14, Ch.~X]{Z} and~\cite[Theorem~3.2.5]{MM08}), one has
\begin{equation}\label{eq:MZ}
\|f - \mathcal{L}_n f\|_{L_p}
\le C_p E_n(f)_{L_\infty}
\le C_{r,p}\,\omega_r\!\left(f,\tfrac1n\right)_{L_\infty},
\end{equation}
where $E_n(f)_X$ denotes the error of best approximation of $f$ by trigonometric polynomials of degree at most $n$ in $X$, and $\omega_r(f,\delta)_X$ is the integral modulus of smoothness of order $r$ (see~\eqref{eqmod1}).

It is well known that, in general, the right-hand side of~\eqref{eq:MZ} cannot be replaced by $\lambda_n E_n(f)_{L_p}$ or $\lambda_n \omega_r(f,1/n)_{L_p}$, even for rapidly increasing sequences $\lambda_n$ and sufficiently smooth functions $f$; see~\cite{BXZ92,WS03}.
Nevertheless, for every $f\in C(\T)$, there holds (see, e.g.,~\cite{H89, O86}) the following estimate:
\begin{equation}\label{eq:Besov}
\|f - \mathcal{L}_n f\|_{L_p}
\le \frac{C_{r,p}}{n^{1/p}}
\int_0^{1/n}
\frac{\omega_r(f,t)_{L_p}}{t^{1/p+1}}\,dt.
\end{equation}
Moreover, this estimate together with the classical inverse approximation theorem, see, e.g.,~\cite[Ch.~7, \S~3]{DL} yields that for $f\in C(\T)$ and $1/p<\alpha<r$,
\begin{equation}\label{eq:rate}
\|f - \mathcal{L}_n f\|_{L_p}=\mathcal{O}(n^{-\alpha})
\quad\Longleftrightarrow\quad
\w_r(f,\d)_{L_p}=\mathcal{O}(\d^\a).
\end{equation}

The property~\eqref{eq:rate} fails for $\alpha\le 1/p$. Even so, meaningful approximation results for $\mathcal{L}_n f$ remain available for functions of low smoothness.
In particular, it was shown in~\cite{Pr84} that if $f$ is of bounded variation on $\T$, then
$\|f - \mathcal{L}_n f\|_{L_p}= \mathcal{O}(n^{-1/p})$.
More general estimates can be formulated in terms of the error of best one-sided approximation $\widetilde{E}_n(f)_{L_p}$ (see~\eqref{er_best_ones}) and the $\tau$-modulus of smoothness\footnote{Recall that the $\tau$-modulus (averaged modulus of smoothness) is defined by
$$
\tau_r(f,\d)_{L_p}:=\|\w_r(f,\cdot,\d)\|_{L_p},\quad
\w_r(f,x,\d):=\sup\{|\Delta_h^r f(t)|:\ t,t+rh\in[x-r\d/2,x+r\d/2]\}.
$$}.
For every bounded function $f$ on $\T$, one has
\begin{equation}\label{eq:tau}
\|f - \mathcal{L}_n f\|_{L_p}
\le C_p \widetilde{E}_n(f)_{L_p}
\le C_{r,p}\tau_r\!\left(f,\tfrac{1}{n}\right)_{L_p}.
\end{equation}
(see~\cite{H89}, \cite[Ch.~8]{SP}; see also~\cite{SS00} and~\cite{BBSV06} for a similar result on $\R$).

Estimate~\eqref{eq:tau} implies the results mentioned above for $\mathcal{L}_n f$. However, it is not sharp in general. Namely, if $\alpha\le 1/p$, then there exist continuous functions $f$ such that
$
\|f - \mathcal{L}_n f\|_{L_p} = \mathcal{O}(n^{-\alpha})
$
does not imply
$
\widetilde{E}_n(f)_{L_p}=\mathcal{O}(n^{-\alpha})$
or
$\tau_r(f,\delta)_{L_p}=\mathcal{O}(\delta^\alpha),
$
see, e.g.,~\cite{KL23}.

These shortcomings were eliminated in the recent work~\cite{KL23} (see also~\cite{K25} for the non-periodic case) by introducing a new semi-discrete modulus of smoothness defined via the Steklov averaging operator $A_h$ (see~\eqref{eq:steklov}):
\begin{equation*}
\Omega_2\!\left(f,\tfrac{1}{n}\right)_p
:= \(\frac1{2n+1} \sum_{k=0}^{2n}|f(t_k) - A_{\frac\pi{2n+1}} f(t_k)|^p\)^{1/p}+\w_2(f,1/n)_p\,.
\end{equation*}
This modulus makes it possible to establish matching direct and inverse approximation theorems for various sampling operators in $L_p(\T)$ and it is well defined for every function for which the corresponding interpolation operator is meaningful.

It was shown in~\cite{KL23} that for any finite-valued function $f\in L_p(\T)$, $1<p<\infty$, and for any $\alpha\in(0,2)$, the following equivalence and convergence criterion hold:
\begin{equation}\label{eq:Omega-equiv}
\|f - \mathcal{L}_n f\|_{L_p(\T)} = \mathcal{O}(n^{-\alpha})
\quad\iff\quad
\Omega_2(f,\tfrac{1}{n})_p = \mathcal{O}(n^{-\alpha}),
\end{equation}
\begin{equation}\label{eq:Omega-conv}
\lim_{n\to\infty}\|f - \mathcal{L}_n f\|_{L_p(\T)} = 0
\quad\iff\quad
\lim_{n\to\infty}\Omega_2\!\left(f,\tfrac{1}{n}\right)_p = 0.
\end{equation}
The paper~\cite{KL23} also provides examples, where $\Omega_2(f,\tfrac{1}{n})_p$ is explicitly computed for several special functions, demonstrating its advantages over the $\tau$-modulus.

One of the main goals of the present paper is to obtain analogues of the results discussed above for general linear sampling operators $G_n$ acting in Banach lattices $X(\T)$ and $X(\R)$, which are not necessarily translation invariant but have uniformly bounded Steklov averages $A_h$.
In such spaces, a natural analogue of the modulus $\Omega_2$ can be introduced by the formula:
\begin{equation*}
\Omega_2^*\left(f,\tfrac{1}{n}\right)_X
:=
\bigg\|\sum_{k=0}^{2n} |f(t_k) - A_{\frac{\pi}{2n+1}}f(t_k)|\chi_{[t_k,t_{k+1}]}\bigg\|_{X(\T)}
+
\|f - A_{\frac{\pi}{2n+1}} f\|_{X(\T)}.
\end{equation*}
In particular, we prove (see Corollaries~\ref{cor1} and~\ref{cor2+}) that a full analogue of~\eqref{eq:Omega-equiv} and~\eqref{eq:Omega-conv} hold in such spaces $X$:
\begin{equation*}
\|f - \mathcal{L}_n f\|_{X(\T)} = \mathcal{O}(n^{-\alpha})
\quad\iff\quad
\Omega_2^*\!\left(f,\tfrac{1}{n}\right)_X = \mathcal{O}(n^{-\alpha}),
\end{equation*}
\begin{equation*}
\lim_{n\to\infty}\|f - \mathcal{L}_n f\|_{X(\T)} = 0
\quad\iff\quad
\lim_{n\to\infty}\Omega_2^*\!\left(f,\tfrac{1}{n}\right)_X = 0.
\end{equation*}

We stress that our contribution is not limited to extending known estimates to general Banach lattices $X$. We also show that the moduli $\Omega_2$ and $\Omega_2^*$ are effective tools for establishing strong converse inequalities for sampling operators; see Section~\ref{secs} and Theorems~\ref{br} and~\ref{brr}. While such inequalities are well studied for various convolution and Kantorovich-type operators (see~\cite{DI93}), they have not, to our knowledge, been addressed for sampling operators, even in classical spaces $L_p(\T)$, $1\le p<\infty$, except for partial results in~\cite[Theorem~4.4]{KP21}, where the converse-type inequalities involve the term $\widetilde{E}_n(f)_{L_p}$, and except for the special converse inequalities obtained in~\cite{AD24} in the context of the approximation of derivatives in $L_p$-norm. Moreover, we improve several results from~\cite{KL23} and \cite{K25} by relaxing restrictive assumptions on sampling operators related to inverse stability inequalities (see~\eqref{a2} and~\eqref{a2r}) and by providing error estimates in certain discrete norms.

The paper is organized as follows.
Section~2 contains notation and auxiliary results on Banach lattices, Marcinkiewicz--Zygmund type inequalities, moduli of smoothness, and errors of best approximations.
In Section~3 we present the main approximation results for periodic functions, including direct and inverse estimates, strong converse inequalities, and illustrative examples.
Section~4 is devoted to approximation problems on the real line and contains analogous results and examples.

Throughout the paper, $c$, $C$, and $C_j$, $j=1,2,\dots$, denote positive constants which may vary from line to line and are independent of the essential parameters, typically $n$, $\s$, and $f$.

\section{Notation and auxiliary results}

\subsection{Banach lattices}

Let $\Omega$ be either the unit circle $\mathbb{T}$, identified with the interval $[-\pi, \pi)$, or the real line $\mathbb{R}$. We denote by $\mu$ the Lebesgue measure on $\mathbb{R}$ and the normalized Lebesgue measure on $\mathbb{T}$.

Let $X = X(\Omega)$ be a Banach function lattice on $\Omega$ with respect to the measure $\mu$, that is, a Banach space of measurable functions on $\Omega$ with the norm $\|\cdot\|_X$ satisfying the following properties:
\begin{itemize}
  \item[$1)$] if $f$ is measurable and $g \in X$ satisfy $|f| \le |g|$, then $f \in X$ and $\|f\|_X \le \|g\|_X$,
  \item[$2)$] whenever $f_n\in X$, with $\sup_n \|f_n\|_X<\infty$, and $0\le f_n \uparrow f$, then $f\in X$ and $\|f_n\|_X \to \|f\|_X$,
  \item[$3)$] if $E \subset \Omega$ is measurable and~$\mu(E) < \infty$, then $\chi_E\in X$ and there exists
  a constant $c_E>0$ such that  $\int_E |f(x)|d\mu(x)\le c_E\|f\|_X$.
\end{itemize}

For clarity, functions in $X(\Omega)$ are regarded as finite-valued and uniquely defined at every point of $\Omega$, rather than as equivalence classes. This convention allows us to apply sampling operators, which require pointwise values of functions, in a consistent manner.

We denote by $X'$ the associate space of $X$ with the norm
\begin{equation*}
  \|g\|_{X'} := \sup_{\|f\|_X \le 1} \bigg|\int_\Omega g(x) \overline{f(x)} \, d\mu(x)\bigg|.
\end{equation*}
Note that $X'$ is also a Banach function lattice on $\Omega$ with the norm $\|\cdot\|_{X'}$ satisfying the above properties 1)--3), see~\cite[Chapter~1.2]{BS88}.

The space $X$ is called \emph{translation-invariant} if for all $t \in \Omega$ and $f \in X$, the function $f(\cdot + t)$ belongs to $X$ and satisfies
$
\|f(\cdot + t)\|_X = \|f\|_X.
$
The space $X$ is called \emph{rearrangement-invariant} if $\|f\|_X = \|g\|_X$ whenever
$f$
 and
$g$
 have the same distribution function, i.e.,
$$
\mu_f(y) = \mu_g(y), \quad \text{where} \quad \mu_h(y): = \mu\{x \in \Omega : |h(x)| \ge y\}.
$$

We say that $f$ belongs to the Sobolev-type space $X^r$, $r \in \mathbb{N}$, if $f^{(r-1)}$ is (locally) absolutely continuous on $\Omega$ and $f^{(r)} \in X$.  We denote $B=B(\Omega)$  the set of all bounded functions on~$\Omega$.

As usual, $\mathcal{T}_n$ denotes the set of all trigonometric polynomials of degree at most $n$.
The class of band-limited functions  $\mathcal{B}_X^\s=\mathcal{B}_X^\s(\R)$ is given by
$$
\mathcal{B}_X^\s:=\left\{\varphi  \in X\cap \mathcal{S}'(\R)\,:\,\supp\;\widehat{\varphi} \subset [-\s,\s]\right\},
$$
where $\widehat{\varphi}$ denotes the Fourier transform of $\varphi$ in the sense of tempered distributions.
For $f \in L_1(\Omega)$, we define its Fourier transform (or Fourier coefficients if $\Omega = \mathbb{T}$) by
$$
\widehat{f}(\xi)=\mathcal{F}f(\xi)
:=
\frac1{2\pi}\int_\Omega f(x)\,e^{-i \xi x}\,dx,
\quad
\xi \in \mathbb{R}
\ \ (\xi \in \mathbb{Z} \text{ if } \Omega = \mathbb{T}).
$$

Throughout the paper, $\Phi_n$ denotes a linear subspace of $X$ depending on a parameter $n\in \N$ (or $n>0$).
We assume that the family $(\Phi_n)_n$ satisfies the nesting property $\Phi_m\subset \Phi_{n}$ for $n\ge m$, with $\Phi_0=\{0\}$.
Moreover, elements of $\Phi_n$ satisfy the Bernstein-type inequality
\begin{equation}\label{ber}
  \|\vp^{(r)}\|_X \le (Bn)^r\|\vp\|_X,\quad r\in \N,\quad\vp\in \Phi_n,
\end{equation}
for some constant $B=B_{X}>0$.
Under certain natural assumptions on  $X$, all classical examples of such subspaces $\Phi_n$, including trigonometric polynomials, band-limited functions, and shift-invariant spaces generated by a smooth function, satisfy these properties.

In the case of translation-invariant Banach lattices $X$, one may also include spline spaces of a given smoothness $s\in \N$.
In this situation, we assume that~\eqref{ber} holds only for $r\le s$.

For a locally integrable function $f$ and a parameter $h > 0$, we define the Steklov averaging operators $A_h$ and $\dot{A}_h$ by
\begin{equation}\label{eq:steklov}
A_h f(x)
:=
\frac{1}{h}
\int_{x - \frac{h}{2}}^{x + \frac{h}{2}} f(t)\,dt,
\quad
\dot{A}_h f(x)
:=
A_h f\!\left(x + \frac{h}{2}\right).
\end{equation}

Most of the results of the present paper are obtained under the following assumption on the Banach lattice $X$.

\begin{assuma}\label{asumpA}
For all $f \in X$ and  $h > 0$, the Steklov average $A_h f$ belongs to $X$ and
$$
\sup_{h>0,\, \|f\|_X\le 1}\|A_h f\|_X<\infty.
$$
\end{assuma}

We emphasize that this assumption is natural in the context of the approximation error estimates studied in this paper, since it guarantees the proper behavior of smoothness characteristics defined via Steklov averaging operators.
At the same time, it is substantially weaker than the boundedness of the Hardy--Littlewood maximal operator and, in particular, allows one to work with spaces such as $L_1(\Omega)$.

We note that if $X$ satisfies Assumption~A, then the associate space $X'$ also satisfies Assumption~A. This follows from the definition of the norm $\|\cdot\|_{X'}$ and the Fubini-Tonelli theorem.

A function $F : \mathbb{R} \to [0, +\infty]$ is called \emph{symmetrically decreasing} if it is even and non-increasing on $\mathbb{R}_+$.
Given a function $K : \mathbb{R} \to \mathbb{C} \cup \{\infty\}$, we denote by $K^*$ a symmetrically decreasing majorant of $K$, that is, a symmetrically decreasing function such that
$$
|K(x)| \le K^*(x)
\quad \text{for all } x \in \Omega.
$$
We denote by $\mathcal{R}$ the class of integrable symmetrically decreasing functions on $\mathbb{R}$, and by $\mathcal{R}^*$ the class of all functions admitting a summable symmetrically decreasing majorant.

For measurable functions $f$ and $g$ on $\Omega$, their convolution is defined by
$$
(f * g)(x)
:=
\int_\Omega f(t)\,g(x - t)\,d\mu(t),
\quad x \in \Omega,
$$
whenever the integral is well defined.

The following important lemma is proved in~\cite{V23}.

\begin{lemma}\label{lemS}
Let $X$ be a Banach lattice satisfying Assumption~A. If $f \in X$ and $K \in \mathcal{R}^*$, then the convolution $f*K$ exists and is finite almost everywhere, belongs to $X$, and satisfies
\begin{equation*}
\|f*K\|_X
\le
C\,\|K^*\|_{L_1(\Omega)}\,\|f\|_X,
\end{equation*}
where the constant $C$ depends only on $X$.
\end{lemma}

There are many Banach lattices that satisfy Assumption~A.
Most classical translation-invariant Banach function spaces enjoy the Fatou property and the uniform boundedness of Steklov averages.
This class includes Lebesgue, Orlicz, Lorentz, Morrey, and more general rearrangement-invariant spaces; see the monographs
\cite{BS88, KPS82, PKJF2013}.
For non-translation-invariant settings the situation is more delicate.
Although the boundedness of Steklov averages follows from the boundedness of the Hardy--Littlewood maximal operator in $X$, this approach typically requires rather strong structural assumptions.
Such issues have been studied for various non-translation-invariant analogues of classical spaces, including weighted and variable exponent cases (see, e.g., \cite{CF13, HH19, L25}).
On the other hand, Steklov averages can be treated directly, which enlarges the class of admissible Banach lattices $X$.
For instance, it is known that Steklov averages are uniformly bounded in weighted spaces $L_{p,w}$, $1 \le p < \infty$, whenever the weight $w$ belongs to the Muckenhoupt class $\mathrm{A}_p$
(see \cite[Section~5.2.1]{St93}).
Analogous results have recently been obtained for periodic unweighted and weighted variable exponent spaces in \cite{V24, V25}.

\subsection{Marcinkiewicz–Zygmund (MZ)-type inequalities in Banach lattices}

The following discretization results in a Banach lattice
$X$ form a central technical component of the paper.

\begin{lemma}\label{lemmz0}
  Let $X(\T)$ be a Banach lattice, and let
$(x_k)_{k=1}^N\subset \T$ be a set of points such that, for some constant $\g>0$,
$$
0<x_{k+1}-x_k\le\frac{\g}{n},\quad k=1,\dots,N,
$$
where $x_{N+1}=2\pi+x_1$. Then, for all $\vp\in \Phi_n$, we have
\begin{equation}\label{TrMZ++}
  \bigg\|\sum_{k=1}^{N}|\vp(x_k)|\chi_{[x_k,x_{k+1})}\bigg\|_{X}\le C\|\vp\|_X,
\end{equation}
  where the constant $C$ is independent of $\vp$, $N$, and $n$. In particular, \eqref{TrMZ++} holds if $X(\T)$ is a Banach lattice satisfying Assumption~A and $\Phi_n=\mathcal{T}_n$.
\end{lemma}

\begin{proof}
The assertion follows directly from \cite[Theorems~2.9 and 2.10]{KT26} together with assumption~\eqref{ber}. In the case $\Phi_n=\mathcal{T}_n$, the validity of~\eqref{ber} is provided by Lemma~\ref{lemm3}(i).
\end{proof}

The next lemma provides an MZ-type inequality with an optimal set of nodes.
As usual, the $n$th partial Fourier sum of $f\in L_1(\T)$ is denoted by
$$
S_n f(x):=\sum_{k=-n}^n \widehat{f}(k) e^{ikx}.
$$

\begin{lemma}\label{lemmz}  {\sc (See~\cite[Theorem~4.1]{KT26})}
Let $X(\T)$ be a Banach lattice such that $\sup_n\|S_n\|_{X\to X}<\infty$.
Then, for all $T\in \mathcal{T}_n$, we have
    \begin{equation*}
      \|T\|_{X}\asymp\bigg\|\sum_{k=0}^{2n} |T(t_k)|\chi_{[t_{k},t_{k+1})}\bigg\|_{X},\quad t_k=\frac{2\pi k}{2n+1},
     \end{equation*}
     where $\asymp$ denotes a two-sided inequality with positive constants independent of $T$ and $n$.
\end{lemma}

In the spaces $X(\R)$, we have the following analogue of Lemma~\ref{lemmz0}.

\begin{lemma}\label{lemmz0r}
  Let $X(\R)$ be a Banach lattice, and let
$(x_k)_{k\in \Z}\subset \R$ be a set of points  such that,
for some constant $\g>0$,
$$
0<x_{k+1}-x_k\le\frac{\g}{\s},\quad k\in \Z,
$$
Then, for all $\vp\in \Phi_\s$, we have
\begin{equation}\label{TrMZ}
  \bigg\|\sum_{k\in\Z}|\vp(x_k)|\chi_{[x_k,x_{k+1})}\bigg\|_{X}\le C\|\vp\|_X,
\end{equation}
  where the constant $C$ is independent of $\vp$ and $\s$. In particular, \eqref{TrMZ} holds if $X(\R)$ is a Banach lattice satisfying Assumption~A and $\Phi_\s=\mathcal{B}_X^\s$.
\end{lemma}

\begin{proof}
The proof is analogous to that of Lemma~\ref{lemmz0}.
\end{proof}

\subsection{Moduli of smoothness and errors of best approximation}

If $X$ is a translation invariant Banach lattice on $\Omega$, then the integral  (classical) modulus of smoothness of a function $f\in X$
of order $r\in \N$ and step $\d>0$ is defined by
\begin{equation}\label{eqmod1}
    \w_r(f,\d)_X:=\sup_{0<h\le \d} \Vert \D_h^r f\Vert_X,
\end{equation}
where
$$
\D_h^r f(x):=\sum_{\nu=0}^r\binom{r}{\nu}(-1)^{\nu} f(x+(r-\nu)h),
$$
$\binom{r}{\nu}=\frac{r (r-1)\dots (r-\nu+1)}{\nu!},\quad \binom{r}{0}=1$. 

In the case of non-translation-invariant spaces $X$, this modulus of smoothness is already not suitable and do not make sense in some particular situations. To consider our problems in more general family of Banach lattices, we will use a special moduli of smoothness, which are based on the averaged operators $A_h$ and $\dot A_h$.

It follows from~\cite[Theorem~4.6 and Corollary~3.1]{V23} and~\cite[Theorem~2.4, Ch.~6]{DL} that, for all $f\in L_p(\Omega)$, $1\le p<\infty$, $r\in \N$, and $h>0$, the following relations hold:
\begin{equation}\label{modav1}
  \sup_{0\le t\le h}\|(I-A_t)^r f\|_{L_p}\asymp \|(I-A_h)^r f\|_{L_p}\asymp \w_{2r}(f,h)_{L_p},
\end{equation}
\begin{equation}\label{modav2}
  \sup_{0\le t\le h}\|(I-\dot A_{t})^r f\|_{L_p}\asymp \|(I-\dot A_{h})^r f\|_{L_p}\asymp \w_r(f,h)_{L_p},
\end{equation}
where $\asymp$ denotes  a two-sided inequality with constants independent of $f$ and $h$.
In particular, these equivalences imply that the quantities in~\eqref{modav1} and~\eqref{modav2} involving the averaging operators $A_h$ and $\dot A_h$ exhibit essentially the same behaviour as the classical moduli of smoothness
in $L_p(\Omega)$, cf. Lemmas~\ref{lemm1}--\ref{lemm3}.  However, unlike the moduli of smoothness, these quantities remain applicable in the setting of non-translation-invariant spaces $X$.

The error of best approximation and the error of best one-sided approximation of $f\in X(\T)$ by polynomials from $\mathcal{T}_n$ are given by
$$
E_n(f)_X:=\inf\{\Vert f-T_n\Vert_X\,:\, T_n\in \mathcal{T}_n\}
$$
and
\begin{equation}\label{er_best_ones}
  \widetilde{E}_n(f)_X:=\inf\{\Vert Q_n-q_n\Vert_X\,:\,  Q_n,q_n \in \mathcal{T}_n,\quad q_n(x)\le f(x)\le Q_n(x)\},
\end{equation}
respectively.

With a slight abuse of notation, the error of best approximation and the error of best one-sided approximation of $f\in X(\R)$ by band-limited functions of order $\s>0$ are given by
$$
{E}_\s(f)_X:=\inf\{\Vert f-g_\s\Vert_X\,:\, g_\s\in \mathcal{B}_X^\s\}
$$
and
$$
\widetilde{{E}}_\s(f)_X:=\inf\{\Vert Q_\s-q_\s\Vert_X\,:\,  Q_\s,q_\s \in \mathcal{B}_X^\s,\quad q_\s(x)\le f(x)\le Q_\s(x)\}.
$$

For the results presented in the next three lemmas, see~\cite{V23}.

\begin{lemma}\label{lemm1}
Let $X$ be a Banach lattice satisfying Assumption~A.
For $f\in X$,\, $h,\l>0$, and $r,s\in \N$, $s\le r$, the following properties hold:

\begin{enumerate}
  \item[(i)]  $\|(I-A_{h})^r f\|_X\le C\|f\|_X$,\\
                $\|(I-\dot A_{h})^r f\|_X\le C\|f\|_X$;
  \medskip

  \item[(ii)]
               $
                \|(I-A_{\lambda h})^r f\|_{X}\le C (1+\l^{2r})\|(I-A_{h})^r f\|_{X},
               $\\
               $
                \|(I-\dot A_{\lambda h})^r f\|_{X}\le C (1+\l^{r})\|(I-\dot A_{h})^r f\|_{X};
               $
  \medskip

  \item[(iii)]
    $
    C^{-1}\|(I-A_{h})^r f\|_{X}\le \|(I-\dot A_{h})^{2r} f\|_{X} \le C\|(I-A_{h})^r f\|_{X};
    $

  \medskip

  \item[(iv)]  if $f\in X^{2s}$, then  $\|(I-A_{h})^r f\|_{X} \le Ch^{2s}  \|(I-A_{h})^{r-s} f^{(2s)}\|_{X}$,\\
                  if $f\in X^s$, then $\|(I-\dot A_{h})^r f\|_{X} \le Ch^{s}  \|(I-\dot A_{h})^{r-s} f^{(s)}\|_{X}$.
\end{enumerate}
In all these inequalities, the constant $C$ is independent of $f$, $h$, and $\l$.
\end{lemma}

The next lemma gives analogues of the direct and inverse approximation theorems in $X$  for trigonometric polynomials and band-limited functions.

\begin{lemma}\label{lemm2}
Let $X$ be a Banach lattice satisfying Assumption~A.
For $f\in X$, $\g>0$, and $r\in \N$, the following properties hold:
\begin{enumerate}

  \item[(i)]  $E_n(f)_X\le C\|(I-A_{\g/n})^r f\|_{X}$,\\
  \medskip
               $E_n(f)_X\le C\|(I-\dot A_{\g/n})^r f\|_{X}$;

  \medskip

  \item[(ii)]  $\|(I-A_{{\g}/{n}})^r f\|_{X}\le \tfrac{C}{n^{2r}}\sum_{k=0}^{n-1} (k+1)^{2r-1}E_k(f)_X$,\\
  \medskip
                 $\|(I-\dot A_{{\g}/{n}})^r f\|_{X}\le \tfrac{C}{n^{r}}\sum_{k=0}^{n-1} (k+1)^{r-1}E_k(f)_X$.
\end{enumerate}
In all these inequalities, the constant  $C$ is independent of $f$ and $n$.
\end{lemma}

Analogues of the Bernstein inequality and Nikolskii-Stechkin-Boas inequality are given in the following lemma.

\begin{lemma}\label{lemm3}
Let $X$ be a Banach lattice satisfying Assumption~A.
For $T\in \mathcal{T}_n$ (or $T\in \mathcal{B}_X^n$ in the case $\Omega=\R$), $\g>0$, and $r\in \N$, the following properties hold:
\begin{enumerate}

   \item[(i)]  $\|T^{(r)}\|_X\le C n^r \|T\|_X$;

  \item[(ii)]  $\|T^{(r)}\|_X\le C n^r \|(I-\dot A_{{\g}/{n}})^r T\|_{X}$.
\end{enumerate}
Moreover, if $f\in X$ and $T$ satisfies $\|f-T\|_X\le 2E_n(f)_X$. Then
   $$
   \|T^{(r)}\|_X\le C n^r \|(I-\dot A_{{\g}/{n}})^r f\|_{X}.
   $$
In all these inequalities, the constant  $C$ is independent of $f$ and $n$.
\end{lemma}

\begin{lemma}\label{lemder}
Let $X$ be a Banach lattice satisfying Assumption~A and let $f\in X^r$, $r\in \N$. Then
\begin{equation}\label{lemdereq}
 \widetilde{E}_n(f)_X\le \frac{C}{(n+1)^r}E_n(f^{(r)})_X,
\end{equation}
where the constant $C$ is independent of $f$ and $n$.
\end{lemma}

\begin{proof}
We first consider the case $\Omega=\T$. We follow the proof of Theorem~8.1 in~\cite{SP}. We have
\begin{equation}\label{8.5}
  f(x)=\frac1{2\pi}\int_{\T}f(t)dt-\frac1{\pi}\int_{\T}B_1(x-t)f'(t)dt,
\end{equation}
where $B_1$ is the Bernoulli-type function
$$
B_1(x):=-\sum_{k=1}^\infty \frac{\sin kx}{k}=\left\{
                                              \begin{array}{ll}
                                                \frac{x-\pi}{2}, & \hbox{$0<x<2\pi$,} \\
                                                0, & \hbox{$x=0$.}
                                              \end{array}
                                            \right.
$$
Consider the polynomials $T_{1,n}\in \mathcal{T}_{n-1}$ and $t_{1,n}\in \mathcal{T}_{n-1}$ defined by the following interpolation conditions:
\begin{equation*}
\begin{split}
   &T_{1,n}(\tfrac{2\pi k}{n})=B_1(\tfrac{2\pi k}{n}),\quad k=1,2,\dots,n-1,\quad T_{1,n}(0)=\tfrac{\pi}{2},\\
   &T_{1,n}'(\tfrac{2\pi k}{n})=B_1'(\tfrac{2\pi k}{n})=\tfrac12,\quad k=1,2,\dots,n-1,
\end{split}
\end{equation*}
\begin{equation*}
\begin{split}
   &t_{1,n}(\tfrac{2\pi k}{n})=B_1(\tfrac{2\pi k}{n}),\quad k=1,2,\dots,n-1,\quad t_{1,n}(0)=\tfrac{\pi}{2},\\
   &t_{1,n}'(\tfrac{2\pi k}{n})=B_1'(\tfrac{2\pi k}{n})=\tfrac12,\quad k=1,2,\dots,n-1.
\end{split}
\end{equation*}
It is easy to see that
\begin{equation*}
  t_{1,n}(x)\le B_1(x)\le T_{1,n}(x),\quad x\in\T.
\end{equation*}
The polynomial
\begin{equation}\label{FTt}
  F_n(x)=T_{1,n}(x)-t_{1,n}(x) \in \mathcal{T}_{n-1}
\end{equation}
satisfies
$$
F_n(\tfrac{2\pi k}{n})=F_n'(\tfrac{2\pi k}{n})=0\quad\text{and}\quad F_n(0)=\pi,
$$
which implies that
$$
F_n(x)=\pi\(\frac{\sin \tfrac{nx}{2}}{n \sin \tfrac{x}{2}}\)^2.
$$

Let $q\in \mathcal{T}_{n-1}$ be such that $E_{n-1}(f')_X\le 2\|f'-q\|_X$. We set
$$
P(x)=\frac1\pi\int_{\T}B_1(x-t)q(t)dt.
$$
Using~\eqref{8.5}, we obtain
\begin{equation*}
  f(x)-P(x)=\frac1{2\pi}\int_{\T}f(t)dt+\frac1{2\pi}\int_{\T}B_1(x-t)(q(t)-f'(t))dt.
\end{equation*}
Denote
\begin{equation*}
  \begin{split}
     &T_\pm(x)=\frac1{\pi}\int_{\T} T_{1,n}(x-t)S_\pm(t)dt,\\
     &t_\pm(x)=\frac1{\pi}\int_{\T} t_{1,n}(x-t)S_\pm(t)dt,\\
   \end{split}
\end{equation*}
where $S_+(t)=(q(t)-f'(t))_+$ and $S_-(t)=q(t)-f'(t)-S_+(t)$. Then
\begin{equation*}
  \begin{split}
     &f(x)\le R_1(x)=\frac1{2\pi}\int_{\T}f(t)dt+P(x)+T_+(x)+t_-(x),\\
     &f(x)\ge R_2(x)=\frac1{2\pi}\int_{\T}f(t)dt+P(x)+T_-(x)+t_+(x).
   \end{split}
\end{equation*}
Using~\eqref{FTt} and the estimate
$$
F_n(x)\le c\min(1,\tfrac{1}{n^2x^2}),\quad x\in \T,
$$
and applying Lemma~\ref{lemS}, we obtain
\begin{equation*}
  \begin{split}
     \|R_1-R_2\|_X&=\|T_+ + t_- - T_- - t_+\|_X\\
                  &=\bigg\|\frac1\pi \int_{\T}(T_{1,n}(x-t)-t_{1,n}(x-t))|q(t)-f'(t)|dt\bigg\|_X\\
                  &=\bigg\|\frac1\pi \int_{\T}F_n(t)|q(x+t)-f'(x+t)|dt\bigg\|_X\\
                  &\le\frac c\pi\int_{\T}\min(1,\tfrac{1}{n^2x^2})\|f'-q\|_X\,dt\le \frac CnE_{n-1}(f')_X.
   \end{split}
\end{equation*}
This proves~\eqref{lemdereq} for $r=1$. The general case for $\Omega=\T$ follows by induction.

The same approach applies to  $\Omega=\R$, using the required estimates for an integral analogue of the Bernoulli function from~\cite{T89}, which also contains the proof of~\eqref{lemdereq} for $L_p(\R)$.
\end{proof}

\section{Main results in the periodic case}\label{main}

In this section, we set $X=X(\T)$ and consider a sequence of linear sampling operators
\begin{equation*}
  G_nf(x):=\sum_{k=1}^{m_n} f(x_{k,n})\vp_{k,n}(x),\quad n\in \N,
\end{equation*}
where $\vp_{k,n} \in \Phi_n$, $k=1,\dots, m_n$, are $2\pi$-periodic and $\mathcal{X}_n=(x_{k,n})_{k=1}^{m_n}\subset \T$ is a set of points for which there exist positive constants $\g$ and $\g'$ satisfying
\begin{equation}\label{g}
  \frac{\g}{n}\le x_{k+1,n}-x_{k,n}\le \frac{\g'}{n},\quad k=1,\dots,m_n,
\end{equation}
with the convention $x_{m_n+1,n}=2\pi+x_{1,n}$. For simplicity, we set $x_k=x_{k,n}$ and $G_0:=0$.

We introduce the following semi-norm
$$
\|f\|_{X_n}:=\bigg\|\sum_{k=1}^{m_n}|f(x_k)|\chi_{[x_k,x_{k+1})}\bigg\|_{X}.
$$
In this regard, we recall that we do not identify functions that are equal almost everywhere. That is,
each function in $X$ is finite-valued and determined by its values at each point in $\T$. Note that that in the case $X=L_p(\T)$, the semi-norm $\|f\|_{X_n}$ can be written by
$$
\|f\|_{X_{n,p}}:=\|f\|_{\ell_p(\mathcal{X}_n)}=\(\sum_{k=1}^{m_n}|f(x_k)|^p(x_{k+1}-x_k)\)^\frac1p.
$$

To formulate the main results, we use the following assumptions on
$G_n$, involving a fixed parameter $s\in \N$:
\begin{equation}\label{a1}\tag{$A_1$}
  \|G_nf\|_X\le K_1\|f\|_{X_n},\quad f\in X,\quad n\in \N,
\end{equation}

\begin{equation}\label{a2}\tag{$A_2$}
  K_2\|f\|_{X_n}\le \|G_nf\|_X,\quad f\in X,\quad n\in \N,
\end{equation}

\begin{equation}\label{a3}\tag{$A_3$}
\|T-G_nT\|_X\le K_3n^{-s}\|T^{(s)}\|_X, \quad T\in \mathcal{T}_n,\quad n\in \N.
\end{equation}

\smallskip

\noindent Here $K_1,K_2>0$ depend only on $X$ and $K_3>0$ depends only on $s$ and $X$. 

\smallskip

A typical example of sampling operators satisfying these three assumptions simultaneously is the classical Lagrange interpolation polynomial $\mathcal{L}_n f$.
Indeed, inequality~\eqref{a3} is clearly satisfied, whereas \eqref{a1} and \eqref{a2} follow from Lemma~\ref{lemmz}, provided that the Fourier partial sums $S_n$ are uniformly bounded in $X$. Furthermore, if $X=L_p(\T)$, then inequalities~\eqref{a1}–\eqref{a3} also hold for the Lagrange interpolation operator with non-uniformly distributed nodes, subject to certain restrictions depending on $p$; see~\cite{CZ99} and~\cite{KL23}. Additional examples of sampling operators $G_n$ satisfying properties~\eqref{a1}–\eqref{a3} can be found in Section~\ref{secext}, as well as in~\cite{KL23}, \cite{K25}, and~\cite{CL25}, where such operators are studied in both periodic and non-periodic settings.

\subsection{Estimates in  terms of special moduli of smoothness}\label{seci}

\begin{theorem}\label{thst}
Let $X$ be a Banach lattice satisfying Assumption~A, and let $r, s \in \N$ with $2r\ge s$.
Suppose that $G_n$ satisfies assumptions~\eqref{a1} and~\eqref{a3}. Then, for all  $f\in X$, we have
\begin{equation}\label{st1}
  \|f-G_n f\|_X\le C\(\|(I-\dot A_{\g/n})^s f\|_{X}+\|(I-A_{\g/n})^r f\|_{X_n}\).
\end{equation}
  If, additionally, $G_n\,: X_n\, \to \mathcal{T}_{n}$, then
\begin{equation}\label{st1t}
  \|f-G_n f\|_{X_n}\le C\(\|(I-\dot A_{\g/n})^s f\|_{X}+\|(I-A_{\g/n})^r f\|_{X_n}\),
\end{equation}
 In both estimates, the constant $C$ is independent of $f$ and $n$.
\end{theorem}

To prove this theorem, we require two auxiliary lemmas. The first is of a technical nature, whereas the second is of independent interest.

\begin{lemma}\label{bxn}
Let $X$ be a Banach lattice satisfying Assumption~A, and let $0<h\le \g/n$. Then, for all $f\in X$, we have
\begin{equation*}
 \|A_{h} f\|_{X_n}\le \tfrac{\g+\g'}{nh}\|f\|_X,
\end{equation*}
where $\g$ and $\g'$ are defined in~\eqref{g}.
\end{lemma}

\begin{proof}
Let $t\in [x_k, x_{k+1})$ for a fixed $k$. Noting that
$$
[x_k-\tfrac{\g}{2n}, x_k+\tfrac{\g}{2n}] \subset [t-\tfrac{\g+\g'}{2n}, t+\tfrac{\g+\g'}{2n}],
$$
we get
$$
\sum_{k=1}^{m_n}|A_{\g/n} f(x_k)|\chi_{[x_k,x_{k+1})}(t)\le (1+\tfrac{\g'}{\g})A_{(\g+\g')/{n}}|f|(t)\quad\text{for any}\quad t\in \T.
$$
Thus, applying Assumption~A, we obtain
\begin{equation*}
\begin{split}
   \|A_{h} f\|_{X_n}&\le \tfrac{\g}{nh}\|A_{\g/n} |f|\|_{X_n}\le \tfrac{\g}{nh}(1+\tfrac{\g'}{\g})\|A_{(\g+\g')/{n}}|f|\|_{X} \le \tfrac{\g+\g'}{nh}\||f|\|_X,
\end{split}
\end{equation*}
which proves the lemma.
\end{proof}

\begin{lemma}\label{lemt}
Let $X$ be a Banach lattice satisfying Assumption~A, and let $r,s\in \N$ with $2r\ge s$. Then, for all $f\in X$ and $T\in \mathcal{T}_n$ satisfying $\|f-T\|_X\le 2E_n(f)_X$, we have
\begin{equation*}
 \|f-T\|_{X_n}\le C\(\|(I-\dot A_{\g/n})^s f\|_{X}+\|(I-A_{\g/n})^r f\|_{X_n}\),
\end{equation*}
where the constant $C$ is independent of $f$ and $n$.
\end{lemma}

\begin{proof}
  Here and in what follows in the paper, we denote
\begin{equation}\label{fh}
  g_h = g_{h,r} := \sum_{k=1}^r (-1)^{k+1}\binom{r}{k} A_h^k g = g-(I-A_h)^rg\quad\text{and}\quad h=\g /n
\end{equation}
for any locally integrable function $g$.

   We have
  \begin{equation}\label{k2-}
    \begin{split}
       \|f-T\|_{X_n}&\le \|f-f_h\|_{X_n}+\|f_h-T_h\|_{X_n}+\|T_h-T\|_{X_n}.
    \end{split}
  \end{equation}
  Applying Lemmas~\ref{bxn} and~\ref{lemm2}(i), we get
  \begin{equation}\label{k3}
    \begin{split}
       \|f_h-T_h\|_{X_n}\le C\|f-T\|_X\le C\|(I-\dot A_{\g/n})^s f\|_{X}.
    \end{split}
  \end{equation}
  Further, by Lemmas~\ref{lemmz0}, \ref{lemm1}(i), (iii), and~\ref{lemm2}(i),  we obtain
  \begin{equation}\label{k4}
    \begin{split}
       \|T_h-T\|_{X_n}&\le C\|T_h-T\|_{X}=C\|(I-A_{\g/n})^r T\|_{X}\\
       &\le C\(\|(I-A_{\g/n})^r (f-T)\|_{X}+\|(I-A_{\g/n})^r f\|_{X}\)\\
       &\le C\(\|f-T\|_{X}+\|(I-\dot A_{\g/n})^{2r} f\|_{X}\)\\
       &\le C\|(I-\dot A_{\g/n})^s f\|_{X}.
    \end{split}
  \end{equation}
Combining~\eqref{k2-}--\eqref{k4}, we prove the lemma.
\end{proof}

\begin{proof}[Proof of Theorem~\ref{thst}]
Let $f_h$ and $h$ be defined by~\eqref{fh}. Inequality~\eqref{a1} yields
\begin{equation}\label{mo1}
\begin{split}
  \|f-G_n f\|_X&\le \|f-f_h\|_X+\|f_h-G_n f_h\|_X+\|G_n(f-f_h)\|_X\\
  &\le \|f-f_h\|_{X}+\|f_h-G_n f_h\|_X+C\|f-f_h\|_{X_n}.\\
\end{split}
\end{equation}
Let $T\in \mathcal{T}_n$ be such that $\|f-T\|_X\le 2E_n(f)_X$. Then, similarly as in~\eqref{mo1}, we have
\begin{equation}\label{zv1x}
  \begin{split}
      \|f_h-G_n f_h\|_X
       \le \|f_h -T_{h}\|_{X}+\|T_{h}-G_n T_{h}\|_X+C\|f_h -T_{h}\|_{X_n}.
   \end{split}
\end{equation}
Applying Assumption~A and Lemmas~\ref{bxn} and~\ref{lemm2}(i), we obtain
\begin{equation}\label{zv2}
  \begin{split}
     \|f_h -T_{h}\|_{X}+\|f_h -T_{h}\|_{X_n}\le C\|f-T\|_X\le C\|(I-\dot A_{\g/n})^s f\|_{X}.
   \end{split}
\end{equation}
By~\eqref{a3}, Assumption~A, and Lemma~\ref{lemm3}(ii), we get
\begin{equation*}
  \begin{split}
      \|T_{h}-G_n T_{h}\|_X &\le \frac{C}{n^s}\|T_{h}^{(s)}\|_X=\frac{C}{n^s}\|(T^{(s)})_h\|_X\\
      &\le \frac{C}{n^s}\|T^{(s)}\|_X
      \le C\|(I-\dot A_{{\g}/{n}})^s f\|_{X},
   \end{split}
\end{equation*}
which, together with~\eqref{zv1x} and \eqref{zv2}, yields
\begin{equation}\label{mo2}
  \|f_h-G_n f_h\|_X\le C\|(I-\dot A_{{\g}/{n}})^s f\|_{X}.
\end{equation}
Combining~\eqref{mo1} and~\eqref{mo2}, we arrive at~\eqref{st1}.

\smallskip

We now prove~\eqref{st1t}. We have
  \begin{equation}\label{t1}
     \|f-G_n f\|_{X_n}\le \|f-T\|_{X_n}+\|T-G_n T\|_{X_n}+\|G_n(T-f)\|_{X_n}.
  \end{equation}
By Lemma~\ref{lemmz0} and assumptions~\eqref{a3} and~\eqref{a1}, we obtain
\begin{equation}\label{t3}
  \begin{split}
    \|T-G_n T\|_{X_n}+\|G_n(T-f)\|_{X_n}&\le C\(\|T-G_n T\|_{X}+\|G_n(T-f)\|_{X}\)\\
    &\le C\(n^{-s}\|T^{(s)}\|_X+\|T-f\|_{X_n}\).
  \end{split}
\end{equation}
Finally, \eqref{t1} and~\eqref{t3} together with Lemmas~\ref{lemm3} and~\ref{lemt}, yields~\eqref{st1t}.
\end{proof}

Our next goal is to establish inverse estimates for operators~$G_n$.

\begin{proposition}\label{prinv}
Let $X$ be a Banach lattice satisfying Assumption~A, and let $r, s \in \N$ with $2r\ge s$.
For all $f\in X$, the following statements hold:
\begin{itemize}
  \item[(i)]  If $G_n$ satisfies assumptions~\eqref{a1}, \eqref{a2}, and~\eqref{a3}, then
\begin{equation}\label{st1inv}
  C_1\|(I-A_{\g/n})^r f\|_{X_n}-C_2\|(I-\dot A_{\g/n})^s f\|_{X}\le  \|f-G_n f\|_X.
\end{equation}
  \item[(ii)]  If $G_n$ satisfies assumptions~\eqref{a1} and~\eqref{a3}, then
\begin{equation}\label{st1invt}
 C_1\|(I-A_{\g/n})^r f\|_{X_n}-C_2\|(I-\dot A_{\g/n})^s f\|_{X}\le  \|f-G_n f\|_X+\|f-G_n f\|_{X_n}.
\end{equation}
\end{itemize}
 In both estimates, the constants $C_1$ and $C_2$ are independent of $f$ and $n$.
\end{proposition}

\begin{proof}  Let $f_h$ and $h$ be defined by~\eqref{fh}. For $G_n$ satisfying~\eqref{a1} and~\eqref{a3}, we have
\begin{equation}\label{mo2++}
  \|f_h-G_n f_h\|_X\le C\|(I-\dot A_{{\g}/{n}})^s f\|_{X},
\end{equation}
see inequality~\eqref{mo2} in the proof of Theorem~\ref{thst}.

$\textbf{(i)}$
Using assumption~\eqref{a2} together with \eqref{mo2++} and Lemma~\ref{lemm1}(iii), (i), we obtain
\begin{equation*}
  \begin{split}
     K_2\|(I-A_{\g/n})^r f\|_{X_n}&=K_2\|f-f_h\|_{X_n}\le \|G_n (f-f_h)\|_X\\
     &\le \|f-f_h\|_X+\|f_h-G_n f_h\|_X+\|f-G_n f\|_X\\
     &\le C\(\|(I-A_{\g/n})^r f\|_{X}+\|(I-\dot A_{\g/n})^s f\|_{X}\)+\|f-G_n f\|_X\\
     &\le C\|(I-\dot A_{\g/n})^s f\|_{X}+\|f-G_n f\|_X,
   \end{split}
\end{equation*}
which proves~\eqref{st1inv}.

\smallskip

$\textbf{(ii)}$
Applying Lemma~\ref{lemmz0}, we get
\begin{equation}\label{prt1}
  \begin{split}
     \|(I-A_{\g/n})^r f\|_{X_n}&=\|f-f_h\|_{X_n}\\
     &\le \|G_n(f-f_h)\|_{X_n}+\|f-f_h-G_n(f-f_h)\|_{X_n}\\
     &\le C\|G_n(f-f_h)\|_{X}+\|f-G_nf\|_{X_n}+\|f_h-G_n f_h\|_{X_n}.
   \end{split}
\end{equation}
By Lemma~\ref{lemm1}(iii) and inequality~\eqref{mo2}, we obtain
\begin{equation}\label{prt2}
  \begin{split}
     \|G_n(f-f_h)\|_{X}&\le \|G_n f-f\|_{X}+\|f-f_h\|_{X}+\|f_h-G_n f_h\|_{X}\\
     &\le \|G_n f-f\|_{X}+C\|(I-\dot A_{\g/n})^s f\|_{X}.
  \end{split}
\end{equation}
Furthermore,  using~\eqref{st1t}, Assumption~A, Lemma~\ref{bxn}, and Lemma~\ref{lemm1}(iii), (i), we have
\begin{equation}\label{prt3}
  \begin{split}
     \|f_h-G_n f_h\|_{X_n}&\le C\(\|(I-\dot A_{\g/n})^s f_h\|_{X}+\|(I-A_{\g/n})^r f_h\|_{X_n}\)\\
     &=C\(\|((I-\dot A_{\g/n})^s f)_h\|_{X}+\|((I-A_{\g/n})^r f)_h\|_{X_n}\)\\
     &\le C\(\|(I-\dot A_{\g/n})^s f\|_{X}+\|(I-A_{\g/n})^r f\|_{X}\)\\
     &\le C\|(I-\dot A_{\g/n})^s f\|_{X}.
  \end{split}
\end{equation}
Combining~\eqref{prt1}, \eqref{prt2}, and~\eqref{prt3} completes the proof of ~\eqref{st1invt}.
\end{proof}

\begin{corollary}\label{cor1}
Assume that the set of all trigonometric polynomials is dense in $X$. For  all $f\in B(\T)$, the following statements hold:

\begin{itemize}
  \item[(i)] If the operators $G_n$ satisfy assumptions~\eqref{a1}, \eqref{a2}, and~\eqref{a3}, then
  $$
    \|f-G_n f\|_X\to 0\quad\text{as } n\to \infty
  $$
if and only if
\begin{equation*}
  \|(I-A_{\g/n})^r f\|_{X_n}\to 0\quad\text{as } n\to \infty.
\end{equation*}

  \item[(ii)]  If the operators $G_n: X_n\, \to \mathcal{T}_{n}$ satisfy assumptions~\eqref{a1} and~\eqref{a3}, then
   $$
   \|f-G_n f\|_X+\|f-G_n f\|_{X_n}\to 0\quad\text{as } n\to \infty
   $$
   if and only if
\begin{equation*}
  \|(I-A_{\g/n})^r f\|_{X_n}\to 0\quad\text{as } n\to \infty.
\end{equation*}
\end{itemize}
\end{corollary}

\begin{proof}
By Lemma~\ref{lemm2}(ii), we have  $\|(I-\dot A_{h})^s f\|_{X}\to 0$ as $h\to 0$.
Using this fact, the assertions of the corollary follow directly from Theorem~\ref{thst} and Proposition~\ref{prinv}.
\end{proof}

The following theorem is an analogue of the classical  Bernstein-type inverse theorem for trigonometric polynomials, cf.~\cite[Ch.~7, Theorem~3.1]{DL} and Lemma~\ref{lemm2}(ii).

\begin{theorem}\label{thinv}
Let $X$ be a Banach lattice satisfying Assumption~A, and let $r, s \in \N$ with $2r\ge s$.
For all $f\in X$, the following statements hold:
\begin{itemize}
  \item[(i)]  If the operators $G_n$ satisfy assumptions~\eqref{a1}, \eqref{a2}, and~\eqref{a3}, then
\begin{equation*}
  \|(I-\dot A_{\g/n})^s f\|_{X}+\|(I-A_{\g/n})^r f\|_{X_n}\le  \frac{C}{n^s}\sum_{k=0}^{n} (k+1)^{s-1}\|f-G_k f\|_X.
\end{equation*}
  \item[(ii)]  If the operators  $G_n$ satisfy assumptions~\eqref{a1} and~\eqref{a3}, then
\begin{equation*}
\begin{split}
    \|(I-&\dot A_{\g/n})^s f\|_{X}+\|(I-A_{\g/n})^r f\|_{X_n}\\
    &\le C\(\|f-G_{n} f\|_X+\|f-G_n f\|_{X_n}+\frac{1}{n^s}\sum_{k=0}^{n} (k+1)^{s-1}\|f-G_{k} f\|_X\).
\end{split}
\end{equation*}
\end{itemize}
 In both estimates, the constant $C$ is independent of $f$ and $n$.
\end{theorem}

\begin{proof}
$\textbf{(i)}$ It follows from~\eqref{st1inv} that
\begin{equation}\label{inv1}
\begin{split}
  \|(I-\dot A_{\g/n})^s f\|_{X}+\|(I-A_{\g/n})^r &f\|_{X_n}\le C\(\|(I-\dot A_{\g/n})^s f\|_{X}+\|f-G_n f\|_X\).
\end{split}
\end{equation}

Denote $E(f,\Phi_k)_X:=\inf_{\vp\in \Phi_k}\|f-\vp\|_X$.
Let $\vp_k\in \Phi_k$ satisfy $\|f-\vp_k\|_X\le 2 E(f,\Phi_k)_X$, $k=0,1,2,\dots.$
We choose $m\in \Z_+$ such that $2^m\le n<2^{m+1}$.
Applying Lemma~\ref{lemm1}(i), (iv), we obtain
\begin{equation}\label{inv2}
\begin{split}
  \|(I-\dot A_{\g/n})^s f\|_{X}&\le \|(I-\dot A_{\g/n})^s (f-\vp_{2^m})\|_{X}+\|(I-\dot A_{\g/n})^s \vp_{2^m}\|_{X}\\
  &\le C\(\|f-\vp_{2^m}\|_X+n^{-s}\|\vp_{2^m}^{(s)}\|_X\).
\end{split}
\end{equation}
Using the equality
$$
\vp_{2^m}=\sum_{\nu=0}^{m} (\vp_{2^\nu}-\vp_{2^{\nu-1}})\quad\text{with the convention}\, \vp_{2^{-1}}:=0
$$
and assumption~\eqref{ber}, we have
\begin{equation}\label{dopp2}
  \begin{split}
     \|\vp_{2^m}^{(s)}\|_{X}&\le \sum_{\nu=0}^{m} \|(\vp_{2^\nu}-\vp_{2^{\nu-1}})^{(s)}\|_{X}\\
     &\le C\sum_{\nu=0}^{m} 2^{s\nu}\|\vp_{2^\nu}-\vp_{2^{\nu-1}}\|_{X}\le C\(\|f\|_p+\sum_{\nu=1}^{m} 2^{s\nu}E(f,\Phi_{2^{\nu-1}})_X\).
   \end{split}
\end{equation}
Thus, from~\eqref{inv2} and~\eqref{dopp2} and the monotonicity of $E(f,\Phi_k)_X$, we get
\begin{equation}\label{inv3}
  \|(I-\dot A_{\g/n})^s f\|_{X}\le Cn^{-s}\sum_{\nu=0}^{n}(\nu+1)^{s-1}E(f,\Phi_\nu)_X.
\end{equation}
Finally, combining~\eqref{inv1} and~\eqref{inv3} together with the inequality $E(f,\Phi_\nu)_X\le \|f-G_\nu f\|_X$, we prove the first statement of the theorem.

\smallskip

$\textbf{(ii)}$  The proof of (ii) follows the same steps as in (i), starting from~\eqref{st1invt}.
\end{proof}

The application of Theorems~\ref{thst} and~\ref{thinv}(i) yields the following corollary.

\begin{corollary}\label{cor2}
  Under the assumptions of Theorem~\ref{thinv}(i), the following properties are equivalent for all  $\a \in (0,s)$:
  \begin{itemize}
    \item[(i)]  $\|f-G_n f\|_X=\mathcal{O}(n^{-\a})$,
    \item[(ii)] $\|(I-A_{\g/n})^r f\|_{X_n}+\|(I-\dot A_{\g/n})^s f\|_{X}=\mathcal{O}(n^{-\a})$.
  \end{itemize}
\end{corollary}

Similarly, the application of Theorems~\ref{thst} and \ref{thinv}(ii) yields the following result.

\begin{corollary}\label{cor2t}
  Under the assumptions  of Theorem~\ref{thinv}(ii), with $G_n: X_n \to \mathcal{T}_{n}$, the following properties are equivalent for all  $\a \in (0,s)$:
  \begin{itemize}
    \item[(i)]  $\|f-G_n f\|_X+\|f-G_n f\|_{X_n}=\mathcal{O}(n^{-\a})$,
    \item[(ii)] $\|(I-\dot A_{\g/n})^s f\|_{X}+\|(I-A_{\g/n})^r f\|_{X_n}=\mathcal{O}(n^{-\a})$.
  \end{itemize}
\end{corollary}

\subsection{Strong converse inequalities for sampling operators}\label{secs}
One of the features of this work is to show that the measure of smoothness $\|(I-\dot A_{\g/n})^s f\|_{X}+\|(I-A_{\g/n})^r f\|_{X_n}$ is also suitable for establishing strong converse inequalities for sampling operators acting on a Banach lattices $X$. Recall that for an approximation method $\mathcal{A}_n$, the strong converse inequality has the form
\begin{equation}\label{eqsc}
  \Omega(f,1/n)_X\le C\|f-\mathcal{A}_n f\|_X,
\end{equation}
where  $\Omega(f,1/n)_X$ is a certain measure  of smoothness compatible with the method $\mathcal{A}_n$
(a classical or modified modulus of smoothness, a $K$-functional, or related quantities) and $C$ is a positive constant independent of $f$ and $n$, see~\cite{DI93}.  Note that, according to~\cite{DI93},  there several types of strong converse inequalities, but the one in~\eqref{eqsc} is mostly considered in the literature. Furthermore, for many approximation processes, it is possible to show that the corresponding direct error estimate for $\|f-\mathcal{A}_n f\|_X$  may hold with the same measure of smoothness $\Omega(f,1/n)_X$, thus providing an equivalence
\begin{equation}\label{eqsc2}
  \|f-\mathcal{A}_n f\|_X\asymp \Omega(f,1/n)_X
\end{equation}
 with positive constants independent of $f$ and $n$.

 For the operators of the convolution-type (Fej\'er means, Bochner–Riesz means, etc.) and certain important operators of different type such as the Bernstein and Kantorovich operators, quasi-projection operators, and families of linear polynomial operators having special kernels, relations of type~\eqref{eqsc2} have been known for a long time, see~\cite{T80}, \cite{DI93}, \cite{To94}, \cite{GZ95}, \cite[Ch.~8]{TB}, \cite{RS08}. However, although for the operators of the convolution type and some Kantorovich type sampling operators, strong converse inequalities, and hence~\eqref{eqsc2}, have been established in various (quasi-)Banach spaces including the Lebesgue spaces $L_p$ with $0<p\le\infty$, see, e.g.,~\cite{K12}, \cite{KP21}, \cite{KS21}, \cite{RS08}, the strong converse inequalities  for sampling operators have been considered so far only in the uniform norm (see, e.g., \cite{AD22}, \cite{To94}, \cite{XZ06}, \cite{Zh09}),  except for partial results in~\cite[Theorem~4.4]{KP21}, where the converse inequalities involve the term $\widetilde{E}_n(f)_{L_p}$
 and except of special converse inequalities obtained in the context of the approximation of derivatives in $L_p$-norm~\cite{AD24}.

In this paper, we identify, for the first time, an appropriate form of strong converse inequalities for sampling operators within a general Banach lattice satisfying Assumption~A, yielding results that are new even in the classical $L_p$-case.
To formulate these results, we need to introduce one additional assumption on $G_n$, involving a fixed parameter $s\in \N$:


{ 
\begin{equation}\label{a4}\tag{$A_4$}
  K_4n^{-s}\|T^{(s)}\|_X \le \|T-G_n T\|_X, \quad T\in \mathcal{T}_n,\quad n\in\N,
\end{equation}

\smallskip

\noindent where  $K_4>0$ depends only on $s$ and  $X$.}

\medskip

For specific conditions on $G_n$ that ensure the validity of~\eqref{a4}, see Proposition~\ref{prex2} below.

\begin{theorem}\label{thC}
  Let $X$ be a Banach lattice satisfying Assumption~A, and let $r, s \in \N$ with $2r\ge s$. Assume that $G_n: X_n\to \mathcal{T}_{n}$ satisfies assumptions~\eqref{a1}, \eqref{a3}, and~\eqref{a4}.
  Then, for all  $f\in X$, we have
  \begin{equation}\label{c0}
    \|f-G_n f\|_{X}+\|f-G_n f\|_{X_n}\asymp\|(I-\dot A_{\g/n})^s f\|_{X}+\|(I-A_{\g/n})^r f\|_{X_n},
  \end{equation}
  where $\asymp$ denotes a two-sided inequality with constants independent of $f$ and $n$.
\end{theorem}

We note that in general the lower estimate in~\eqref{c0} is not valid without the term $\|f-G_n f\|_{X_n}$, see Section~\ref{sec_comp}.

\begin{proof}[Proof of Theorem~\ref{thC}]
Theorem~\ref{thR} proved below yields
\begin{equation}\label{c1}
\begin{split}
   \|(I-A_{\g/n})^r f\|_{X_n}&+\|(I-\dot A_{\g/n})^s f\|_{X}\\
   &\le C\(\|f-G_n f\|_X+\|f-G_n f\|_{X_n}+n^{-s}\|(G_n f)^{(s)}\|_X\).
\end{split}
\end{equation}
Applying assumptions~\eqref{a4} and~\eqref{a1}, we obtain
\begin{equation}\label{c2}
  \begin{split}
    n^{-s}\|(G_n f)^{(s)}\|_X &\le C \|G_n f-G_n G_n f\|_X\\
    &=C \|G_n (f-G_n f)\|_X\le C\|f-G_n f\|_{X_n}.
  \end{split}
\end{equation}
Thus, combining~\eqref{c1} and~\eqref{c2}, we get the lower estimate in~\eqref{c0}.

The corresponding upper estimate follows directly from Theorem~\ref{thst}.
\end{proof}

Note that in the classical case $X=L_p(\T)$,  the two-sided estimate~\eqref{c0} can be written as follows
\begin{equation*}
    \|f-G_n f\|_{L_p(\T)}+\|f-G_n f\|_{\ell_p(\mathcal{X}_n)}\asymp\w_s(f,1/n)_{L_p(\T)}+\|(I-A_{\g/n})^r f\|_{\ell_p(\mathcal{X}_n)}.
  \end{equation*}

\subsection{Special $K$-functionals and their realizations}\label{secr}

Recall that for $f\in X$, $s\in \N$, and $\d>0$, the Peetre  $K$-functional is defined by
\begin{equation*}
  K_s(f,\d)_X:=\inf_{g\in X^s}(\|f-g\|_X+\d^s\|g^{(s)}\|_X).
\end{equation*}
It is well known (see, e.g.,~\cite[Ch.~6]{DL}) that if $X=L_p(\T)$, then
\begin{equation*}
  K_s(f,\d)_{L_p}\asymp \w_s(f,\d)_{L_p},\quad f\in L_p(\T), \quad \d>0.
\end{equation*}
Recently, in~\cite{V23}, this result was generalized to a general Banach lattice $X$ satisfying Assumption~A.  Namely, it was proved that
for all $f\in X$  and $s\in \N$, the following relation holds:
\begin{equation}\label{k1}
  K_s(f,\d)_{X}\asymp \|(I-\dot A_{\d})^s f\|_{X},\quad \d>0,
\end{equation}
where $\asymp$ denotes a two-sided inequality with positive constants independent of $f$ and $\d$.

Following~\cite{KL23}, for $f\in X$, $s\in \N$, and $h>0$, we introduce the "semi-discrete" modification of the Peetre $K$-functional by
\begin{equation}\label{kf}
  K_s(f,h)_{X_n}:=\inf_{g\in X^s}\(\|f-g\|_{X}+\|f-g\|_{X_n}+h^s\|g^{(s)}\|_X\).
\end{equation}

In the case $X=L_p(\T)$, it was proved in~\cite{KL23} that, for all
$f\in L_p(\T)$, $1\le p<\infty$, $r,s\in \N$, and $2r\ge s$,  the following relation holds:
  \begin{equation*}
     \mathcal{K}_s(f,1/n)_{X_{n,p}}\asymp \w_s(f,1/n)_{L_p}+\|(I-A_{1/n})^r f\|_{X_{n,p}}.
  \end{equation*}
Below we extend this result to the case of a general Banach lattice $X$.

\begin{theorem}\label{thK}
 Let $X$ be a Banach lattice satisfying Assumption~A, and let $r, s \in \N$ with $2r\ge s$.
 Then, for all $f\in X$, we have
  \begin{equation*}
    K_s(f,1/n)_{X_n}\asymp \|(I-\dot A_{\g/n})^s f\|_{X}+\|(I-A_{\g/n})^r f\|_{X_n}
  \end{equation*}
  where $\asymp$ denotes a two-sided inequality with constants independent of $f$ and $n$.
\end{theorem}

\begin{proof}
  Let $T\in \mathcal{T}_n$ be such that $\|f-T\|_X\le 2E_n(f)_X$. Using Lemmas~\ref{lemm2}(i), \ref{lemt}, and~\ref{lemm3}, we obtain
  \begin{equation*}
  \begin{split}
     K_s(f,1/n)_{X_n}&\le \|f-T\|_{X}+\|f-T\|_{X_n}+n^{-s}\|T^{(s)}\|_X\\
     &\le C\(\|(I-\dot A_{\g/n})^s f\|_{X}+\|(I-A_{\g/n})^r f\|_{X_n}\).
  \end{split}
  \end{equation*}

We now prove the lower estimate. Let $f_h$ and $h$ be defined by~\eqref{fh}.  For $g\in X^s$, equivalence~\eqref{k1} yields
\begin{equation}\label{kk1}
  \begin{split}
     \|(I-\dot A_{\g/n})^s f\|_{X}\le C\(\|f-g\|_X+n^{-s}\|g^{(s)}\|_X\),
  \end{split}
\end{equation}
while Lemma~\ref{bxn} provides
\begin{equation}\label{kk2}
  \begin{split}
      \|(I-A_{\g/n})^r f\|_{X_n}&=\|f-f_h\|_{X_n}\le \|f-g\|_{X_n}+\|g-g_h\|_{X_n}+\|g_h-f_h\|_{X_n}\\
      &\le \|f-g\|_{X_n}+\|g-g_h\|_{X_n}+C\|f-g\|_{X}.
  \end{split}
\end{equation}
It follows from \eqref{kk1} and \eqref{kk2} that to prove the theorem, it suffices to show
\begin{equation}\label{kk4}
  \|g-g_h\|_{X_n}\le Cn^{-s}\|g^{(s)}\|_X.
\end{equation}
Let $q,Q\in \mathcal{T}_n$ be such that $q(x)\le g(x)\le Q(x)$ and $\|Q-q\|_X\le 2\widetilde{E}_n(g)_X$.
Applying Lemma~\ref{lemmz0}, we obtain
\begin{equation*}
  \begin{split}
     \|g-g_h\|_{X_n}&\le \|Q-q_h\|_{X_n}\le C\|Q-q_h\|_{X}\le C\(\|Q-q\|_{X}+\|q-q_h\|_{X}\).
  \end{split}
\end{equation*}
Lemma~\ref{lemder} yields
\begin{equation*}
  \begin{split}
     \|Q-q\|_{X}\le 2 \widetilde{E}_n(g)_X \le Cn^{-s}\|g^{(s)}\|_X.
  \end{split}
\end{equation*}
At the same time, from Lemmas~\ref{lemm1}(iii), (i), (iv) and~\ref{lemder}, we derive
\begin{equation*}
  \begin{split}
     \|q-q_h\|_{X}&=\|(I-A_{\g/n})^r q\|_{X}\\
     &\le C\|g-q\|_X+\|(I- A_{\g/n})^r g\|_{X}\\
     &\le C\(\|g-q\|_X+\|(I-\dot A_{\g/n})^s g\|_{X}\)\\
     &\le C\(\widetilde{E}_n(g)_X+n^{-s}\|g^{(s)}\|_X\)\le Cn^{-s}\|g^{(s)}\|_X,
  \end{split}
\end{equation*}
which together with the above two formulas yields~\eqref{kk4}.
Now, combining~\eqref{kk1}, \eqref{kk2}, and~\eqref{kk4}, we obtain
  \begin{equation*}
    \|(I-A_{\g/n})^r f\|_{X_n}+\|(I-\dot A_{\g/n})^s f\|_{X}\le C\(\|f-g\|_{X}+\|f-g\|_{X_n}+n^{-s}\|g^{(s)}\|_X\).
  \end{equation*}
To complete the proof, it remains to take the infimum over all $g\in X^s$.
\end{proof}

It is known (see, e.g.,~\cite{HI90}) that in $L_p(\T)$, $1\le p<\infty$, the modulus of smoothness is equivalent to the realization of the $K$-functional:
\begin{equation}\label{Re}
  \w_s(f,1/n)_p\asymp \|f-T\|_p+n^{-s}\|T^{(s)}\|_p,
\end{equation}
where the polynomial $T\in \mathcal{T}_n$ satisfies $\|f-T\|_p\le C\w_s(f,1/n)_p$.
This holds, for example, for polynomials of near best approximation, de la Vall\'ee Poussin means, Riesz means, etc.
For various applications of the realizations of $K$-functionals see, e.g.,~\cite{HI90}, \cite{KT20}--\cite{KT21}.
Below, we give an analogue of equivalence~\eqref{Re} for the sampling operator $G_n$ in a general Banach lattice $X$ satisfying Assumption~A. As a realization of the $K$-functional~\eqref{kf}, we consider the following expression defined using the operators $G_n$:
$$
R_s(f,h)_{X_n}:=\|f-G_n f\|_X+\|f-G_n f\|_{X_n}+h^{s}\|(G_n f)^{(s)}\|_X.
$$

\begin{theorem}\label{thR}
  Let $X$ be a Banach lattice satisfying Assumption~A, and let $r, s \in \N$ with $2r\ge s$. Assume that $G_n:X_n\to \mathcal{T}_{n}$ satisfies assumptions~\eqref{a1} and~\eqref{a3}. Then, for all $f\in X$, we have
  \begin{equation}\label{r0}
  \begin{split}
     R_s(f,1/n)_{X_n}\asymp \|(I-A_{\g/n})^r f\|_{X_n}+\|(I-\dot A_{\g/n})^s f\|_{X},
  \end{split}
  \end{equation}
  where $\asymp$ denotes a two-sided inequality with constants independent of $f$ and $n$.
\end{theorem}

\begin{proof}
Applying Lemmas~\ref{lemm3}(ii) and~\ref{lemm1}(i), we obtain
\begin{equation*}\label{rr2}
  \begin{split}
    n^{-s}\|(G_n f)^{(s)}\|_X&\le C\|(I-\dot A_{\g/n})^s G_n f\|_{X}\\
    &\le C(\|f-G_n f\|_X+\|(I-\dot A_{\g/n})^s f\|_{X}).
  \end{split}
\end{equation*}
Combining this inequality with the estimates of $\|f-G_n f\|_X$ and $\|f-G_n f\|_{X_n}$ from Theorem~\ref{thst}, we prove the upper inequality in~\eqref{r0}.

The lower estimate in~\eqref{r0} follows directly from Theorem~\ref{thK}.
\end{proof}

\subsection{Comparison of the errors $\|f-G_n f\|_X$ and $\|f-G_n f\|_{X_n}$}\label{sec_comp}
In several results above, we obtain error estimates for the combined quantity $\|f-G_n f\|_X+\|f-G_n f\|_{X_n}$ whereas in some other cases the estimates were derived only for $\|f-G_n f\|_X$. If $G_n$ is an interpolation operator at the nodes $(x_k)_{k=1}^{m_n}$, then
$\|f-G_n f\|_{X_n}=0$ for any function $f\in X$ and hence the inequality $\|f-G_n f\|_{X_n}\le\|f-G_n f\|_X$ is trivially satisfied. However, for general sampling operators the quantities $\|f-G_n f\|_X$  and $\|f-G_n f\|_{X_n}$ are not comparable. Indeed, by choosing any nontrivial function $f_n\in X_n$ such that $f_n(x_k)=0$ for all $k=1,\dots,m_n$ we obtain that inequality $\|f-G_n f\|_X\le C\|f-G_n f\|_{X_n}$ cannot hold for any constant $C>0$ independent of $n$. To show that the opposite inequality also fails, we need the following proposition.

\begin{proposition}\label{prex}
  Let $X$ be a translation-invariant Banach lattice with an absolute continuous norm, i.e., $\phi_X(t):=\|\chi_{[0,t]}\|_X\to 0$ as $t\to 0$, and let
  $$
  G_n f(x)=\frac{1}{2n+1}\sum_{j=0}^{2n} f(t_j)K_n(x-t_j),\quad t_j=\frac{2\pi j}{2n+1},
  $$
  where the kernel is given by
  $$
  K_n(x)=\sum_{k=-n}^n c_k e^{ikx}\quad\text{with}\quad c_{k_0}=0 \quad\text{for some}\quad k_0\in \{-n,\dots,n\}.
  $$
  Then the inequality
  \begin{equation}\label{prex1+}
    \|f-G_n f\|_{X_n}\le C\|f-G_n f\|_{X}
  \end{equation}
 does not hold with the constant $C$ independent of $n$. Here, $X_n$ is defined with respect to the points $(t_j)_{j=0}^{2n}$.
\end{proposition}

\begin{proof}
  Since $\phi_X(t)\to 0$ as $t\to 0$, for every $n$ there exists $s_n \in (0,1)$ such that $\phi_X(s_n)\le (2n+1)^{-2}$.
  Let $\psi\in C^\infty(\T)$ be a fixed nonnegative bump function with $\psi(0)=1$ and $\supp \psi \subset (-\tfrac12,\tfrac12)$.
  For every $j$ and $n$, we define
  $$
    \psi_{n,j}(x):=\psi\!\Big(\frac{x-t_j}{w_n}\Big),\quad w_n=\min\left\{s_n, \frac{\pi}{2n+1}\right\},
  $$
   so that the supports of $\psi_{n,j}$ are pairwise disjoint for distinct $j$.
Define the sequence of functions
  $$
    f_n(x):=\sum_{j=0}^{2n} e^{i k_0 t_j}\,\psi_{n,j}(x).
  $$
  By construction, we have that
  $f_n(t_j)=e^{i k_0 t_j}$, $j=0,\dots,2n$,
  which yields
  $$
    \|f_n\|_{X_n}
    = \Big\|\sum_{j=0}^{2n} \chi_{[t_j,t_{j+1})}\Big\|_X
    = \|\chi_{[0,2\pi)}\|_X=:C_X>0.
  $$
  Further, since $|\psi_{n,j}|\le\|\psi\|_\infty\chi_{\operatorname{supp}\psi_{n,j}}$ and $X$ is translation-invariant, we get
  $$
    \|\psi_{n,j}\|_X \le \|\psi\|_\infty \,\|\chi_{\operatorname{supp}\psi_{n,j}}\|_X = \|\psi\|_\infty\,\phi_X(w_n),
  $$
  and hence by the triangle inequality, we obtain
  \begin{equation}\label{fn0}
    \|f_n\|_X \le \sum_{j=0}^{2n} \|\psi_{n,j}\|_X
    \le (2n+1) \|\psi\|_\infty\,\phi_X(w_n)\to 0\quad\text{as } n\to \infty.
  \end{equation}
  At the same time, by linearity, we have
  $$
    G_n f_n(x)
    = \sum_{k=-n}^n c_k e^{ikx}\Big(\frac{1}{2n+1}\sum_{r=0}^{2n} e^{i(k_0-k)t_r}\Big)
    = c_{k_0} e^{i k_0 x}=0.
  $$
   Summarizing the above estimates, we get $\|f_n-G_n f_n\|_{X_n}=\|f_n\|_{X_n}=C_X>0$ for all $n$, whereas~\eqref{fn0} gives
   $\|f_n-G_n f_n\|_X=\|f_n\|_X\to0$. Thus, for any fixed constant $C>0$ there exists $n$ large enough such that
  $$
    \|f_n-G_n f_n\|_{X_n} = C_X > C \|f_n-G_n f_n\|_X,
  $$
  which shows that inequality \eqref{prex1+} cannot hold with a constant $C$ independent of $n$.
\end{proof}

\subsection{Estimates involving the errors of best and best one-sided approximations}

In this section, we extend to general Banach lattices several standard error estimates for sampling operators, using the errors of best and best one-sided approximations. These estimates can be useful when the behaviour of $E_n(f)_X$ and $\widetilde{E}_n(f)_X$ is known for a particular function $f$. However, in general, such estimates are not sharp and apply only to rather restricted classes of functions, compared with the results obtained above in terms of the semi-discrete moduli of smoothness $\|(I-A_{\g/n})^r f\|_{X_n}+\|(I-\dot A_{\g/n})^s f\|_{X}$.

\begin{theorem}\label{thG1}
Let $X$ be a Banach lattice satisfying Assumption~A. For all $f\in B$, the following statements hold:
\begin{itemize}
  \item[(i)]  If $G_n$ satisfies~\eqref{a1} and~\eqref{a3} for some $s\in \N$, then
\begin{equation}\label{g1}
  \|f-G_n f\|_{X}\le C \(\widetilde{E}_n(f)_{X}+\|(I-\dot A_{\g/n})^s f\|_X\).
\end{equation}
  \item[(ii)] If $G_n$ satisfies~\eqref{a1} and  $G_n T = T$ for all $T\in \mathcal{T}_n$, then
\begin{equation}\label{gg1}
  \|f-G_n f\|_{X}\le C \widetilde{E}_n(f)_{X}.
\end{equation}
\end{itemize}
In both estimates, the constant $C$ is independent of $f$ and $n$.
\end{theorem}

\begin{proof}
$\textbf{(i)}$ Let $q_n, Q_n\in \mathcal{T}_n$ be such that $q_n(x)\le f(x)\le Q_n(x)$ for all $x\in \T$
and $\|q_n-Q_n\|_{X}\le 2\widetilde{E}_n(f)_{X}$. By~\eqref{a1} and~\eqref{a3}, we obtain
\begin{equation}\label{g2}
  \begin{split}
     \|f-G_n f\|_{X}&\le \|f-q_n\|_X+\|q_n-G_n q_n\|_X+\|G_n(q_n-f)\|_X\\
     &\le 2\widetilde{E}_n(f)_X+K_3n^{-s}\|q_n^{(s)}\|_X+K_1\|f-q_n\|_{X_n}.
   \end{split}
\end{equation}
 Applying Lemmas~\ref{lemm3}(ii) and~\ref{lemm1}(i), we have
\begin{equation}\label{g3}
\begin{split}
    n^{-s}\|q_n^{(s)}\|_X &\le C \|(I-\dot A_{\g/n})^s q_n\|_X\\
  &\le C\(\|(I-\dot A_{\g/n})^s (q_n-f)\|_X+\|(I-\dot A_{\g/n})^s f\|_X\)\\
  &\le C\(\|f-q_n\|_X+\|(I-\dot A_{\g/n})^s f\|_X\).
\end{split}
\end{equation}
Next, by Lemma~\ref{lemmz0}, we get
\begin{equation*}
  \|f-q_n\|_{X_n}\le \|Q_n-q_n\|_{X_n}\le C\|Q_n-q_n\|_X,
\end{equation*}
which together with \eqref{g2} and~\eqref{g3} yields~\eqref{g1}.

\smallskip

$\textbf{(ii)}$ The proof of~\eqref{gg1} follows the same argument  as the proof of estimate~\eqref{g1}.
\end{proof}

\begin{remark}\label{remG}
  Note that inequality~\eqref{gg1} remains valid even without assuming that $X$ satisfies Assumption~A.
  In this case, however, one needs to assume that the Bernstein inequality~\eqref{ber} holds in $X$ for $\Phi_n = \mathcal{T}_n$.
\end{remark}

\begin{corollary}\label{thG1cor}
Let $X$ be a Banach lattice satisfying Assumption~A and $s\in \N$. For all $f\in X^s$, the following statements hold:
\begin{itemize}
  \item[(i)]  If $G_n$ satisfies the assumptions of Theorem~\ref{thG1}(i), then
\begin{equation}\label{l2-}
  \|f-G_nf\|_{X}\le \frac{C}{n^s} \|f^{(s)}\|_{X}.
\end{equation}

  \item[(ii)] If $G_n$ satisfies the assumptions of Theorem~\ref{thG1}(ii), then
\begin{equation}\label{l2}
  \|f-G_nf\|_{X}\le \frac{C}{n^s} {E}_n(f^{(s)})_{X}.
\end{equation}
\end{itemize}
In both estimates, the constant $C$ is independent of $f$ and $n$.
\end{corollary}

\begin{proof}
$\textbf{(i)}$ Inequality~\eqref{l2-} follows from~\eqref{g1} and Lemmas~\ref{lemder} and~\ref{lemm1}(iv).

$\textbf{(ii)}$ Inequality~\eqref{l2} follows from~\eqref{gg1} and Lemma~\ref{lemder}.
\end{proof}

\begin{corollary}\label{thL1}
Let $X$ be a Banach lattice. Assume that $\sup_n\|S_n\|_{X\to X}<\infty$.
Then, for all $f\in B(\T)$, we have
\begin{equation*}
  \|f-\L_nf\|_{X}\le C \widetilde{E}_n(f)_{X}.
\end{equation*}
If, additionally, $X$ satisfies Assumption~A and $f\in X^r$, then
\begin{equation*}
  \|f-\L_nf\|_{X}\le \frac{C}{n^r} {E}_n(f^{(r)})_{X}.
\end{equation*}
In the above inequalities the constant $C$ is independent of $f$ and $n$.
\end{corollary}

\begin{proof}
It is well-known that $\mathcal{L}_n(T)=T$ for all $T\in \mathcal{T}_n$. Thus, by Lemma~\ref{lemmz}, the operator $\mathcal{L}_n$ satisfies~\eqref{a1} and, hence, applying Theorem~\ref{thG1}(ii) and Remark~\ref{remG}, we prove the corollary.
\end{proof}

To formulate the next theorem, we introduce the dilation operator $\d_r$ acting on functions $f\in X$ by
$$
\d_r f(x)=f(rx),\quad 0<r<1.
$$
An important quantity in what follows is the operator norm $\|\d_{r}\|_{X\to X}$.
In the classical settings of the Lebesgue and Lorentz spaces on the circle,
$X=L_p(\T)$ and
$X=L_{p,q}(\T)$\footnote{The norm in
$L_{p,q}(\mathbb{T})$ is given by
$\|f\|_{L_{p,q}}=\Big(\frac1{2\pi}\int_{\T} (t^{1/p} f^*(t))^q \,\frac{dt}{t}\Big)^{1/q}$ for $1\le q<\infty$ and $\|f\|_{L_{p,\infty}}=\sup_{t\in \T} t^{1/p} f^*(t)$ for $q=\infty$.}, $1\le p<\infty$, $1\le q\le \infty$,
it is straightforward to verify that
$$
\|\d_r\|_{X\to X}=r^{-1/p},\quad 0<r<1.
$$
For the Orlicz spaces $L_\Phi(\T)$ equipped with the Luxemburg norm $\|\cdot\|_{L_\Phi}$\footnote{
The Luxemburg norm is given by
$\|f\|_{L_\Phi}=\inf\{\lambda>0:\frac1{2\pi}\int_{\mathbb{T}}\Phi(|f(x)|/\lambda)\,dx\le 1\}$.}, where $\Phi$ is a  Young function, a direct calculation shows that
$$
\|\d_r\|_{L_\Phi(\T)\to L_\Phi(\T)}=\sup_{t>0}\frac{\Phi^{-1}(t)}{\Phi^{-1}(rt)},
$$
see also~\cite{B71}. For general rearrangement-invariant spaces, some estimates of the norm of the dilation operator can be found, e.g., in~\cite{BS88} and~\cite{KPS82}.

\begin{theorem}\label{thriG}
  Let $X$ be a rearrangement-invariant Banach lattice. Assume that $G_n$ satisfies assumptions~\eqref{a1} and~\eqref{a3} for some $s\in \N$.
Then, for all $f\in X$, we have
  \begin{equation}\label{ri01}
    \|f-G_n f\|_X\le C\(\w_s(f,1/n)_X+\sum_{\nu=1}^\infty \|\d_{2^{-\nu}}\|_{X\to X} E_{2^{\nu-1}n}(f)_X\),
  \end{equation}
where the constant $C$ is independent of $f$ and $n$. Moreover, if additionally, $G_n T=T$ for all $T\in \mathcal{T}_n$, then the above estimate holds without the modulus $\w_s(f,1/n)_X$ on the right-hand side.
\end{theorem}

\begin{proof}
Without loss of generality we can assume that $E_n(f)_X\to 0$ as $n\to \infty$, since  otherwise the left hand side of~\eqref{ri01} is infinite and hence the statement of the theorem is trivial.
First, we show that for each $\nu\in \N$ the following inequality holds:
\begin{equation}\label{ri1}
  \|U_{2^\nu n}\|_{X_n}\le C\|\d_{2^{-\nu}}\|_{X\to X}\|U_{2^\nu n}\|_{X},\quad U_{2^\nu n}\in \mathcal{T}_{2^\nu n},
\end{equation}
where the constant $C$ is independent of $\nu$ and $n$. We have
\begin{equation}\label{ri2}
  \begin{split}
    \|U_{2^\nu n}\|_{X_n}&=\bigg\|\sum_{k=1}^{m_n}|U_{2^\nu n}(x_k)|\chi_{[x_k, x_{k+1})}\bigg\|_X\\
    &\le \|\d_{2^{-\nu}}\|_{X\to X}\bigg\|\sum_{k=1}^{m_n}|U_{2^\nu n}(x_k)|\chi_{[2^{-\nu}x_k, 2^{-\nu}x_{k+1})}\bigg\|_X.
  \end{split}
\end{equation}
Since $\mathcal{X}_n=(x_k)_{k=1}^{m_n}$ satisfies~\eqref{g}, we can construct a set of nodes $\mathcal{Y}_{n,\nu}=(y_\ell)_{\ell=1}^{2^\nu m_n}$ with the following properties:
$$
-\pi\le y_1<y_2<\dots<y_{2^\nu m_n}<\pi,
$$
$$
(x_k)_{k=1}^{m_n}\cup (2^{-\nu}x_k)_{k=1}^{m_n}\subset \mathcal{Y}_{n,\nu},
$$
and there exist positive constants $\tau$ and  $\tau'$ such that
$$
\frac{\tau}{2^\nu n}\le y_{\ell+1}-y_\ell\le \frac{\tau'}{2^\nu n},\quad \ell=1,\dots, 2^\nu m_n,
$$
where $y_{2^\nu m_n+1}=2\pi+y_1$.
Thus, noting that $X$ satisfies Assumption~A (being rearrangement-invariant) and applying Lemma~\ref{lemmz0}, we derive from~\eqref{ri2} the estimate
\begin{equation*}
  \begin{split}
     \|U_{2^\nu n}\|_{X_n}&\le\|\d_{2^{-\nu}}\|_{X\to X}\bigg\|\sum_{\ell=1}^{2^\nu m_n}|U_{2^\nu n}(y_\ell)|\chi_{[y_\ell, y_{\ell+1})}\bigg\|_X\\
     &\le C \|\d_{2^{-\nu}}\|_{X\to X}\|U_{2^\nu n}\|_X,
   \end{split}
\end{equation*}
which yields~\eqref{ri1}.

We now prove~\eqref{ri01}. Let $T_\mu\in \mathcal{T}_\mu$ be such that $\|f-T_\mu\|_X\le 2E_\mu(f)_X$ for each $\mu\in \N$.  By assumptions~\eqref{a3} and~\eqref{a1}, we obtain
\begin{equation}\label{ri3}
\begin{split}
  \|f-G_n f\|_X&\le \|f-T_n\|_X+\|T_n-G_n T_n\|_X+\|G_n(T_n-f)\|_X\\
  &\le \|f-T_n\|_X+K_3n^{-s}\|T_n^{(s)}\|_X+K_1\|f-T_n\|_{X_n}.
\end{split}
\end{equation}
By Lemmas~\ref{lemm2}(i) and~\ref{lemm3} together with~\eqref{modav2}, we have
\begin{equation}\label{ri4}
  \|f-T_n\|_X+n^{-s}\|T_n^{(s)}\|_X\le C\w_s(f,1/n)_X.
\end{equation}
To estimate $\|f-T_n\|_{X_n}$, we use the following representation:
$$
f(x)-T_n(x)=\sum_{\nu=1}^\infty T_{2^{\nu}n}(x)-T_{2^{\nu-1}n}(x),
$$
which together with~\eqref{ri1} yields
\begin{equation}\label{ri5}
\begin{split}
    \|f-T_n\|_{X_n}&\le \sum_{\nu=1}^\infty \|T_{2^{\nu}n}-T_{2^{\nu-1}n}\|_{X_n}\\
&\le C\sum_{\nu=1}^\infty \|\d_{2^{-\nu}}\|_{X\to X}\|T_{2^{\nu}n}-T_{2^{\nu-1}n}\|_X\\
&\le C\sum_{\nu=1}^\infty \|\d_{2^{-\nu}}\|_{X\to X}E_{2^{\nu-1}n}(f)_X.
\end{split}
\end{equation}
Finally, combining~\eqref{ri3}, \eqref{ri4}, and~\eqref{ri5}, we prove~\eqref{ri01}.
\end{proof}

Applying Theorem~\ref{thriG} and the same arguments as in the proof of Corollary~\ref{thL1}, we obtain the following result, which generalizes~\eqref{eq:Besov} to the case of a general rearrangement-invariant Banach lattice $X$.

\begin{corollary}\label{thriL}
  Let $X$ be a rearrangement-invariant Banach lattice. Assume that $\sup_n\|S_n\|_{X\to X}<\infty$. Then, for all $f\in X$, we have
  \begin{equation*}
    \|f-\mathcal{L}_n f\|_X\le C\sum_{\nu=1}^\infty \|\d_{2^{-\nu}}\|_{X\to X} E_{2^{\nu-1}n}(f)_X,
  \end{equation*}
where the constant $C$ is independent of $f$ and $n$.
\end{corollary}

\subsection{Examples}\label{secext}

At the beginning of Section~\ref{main}, we already mentioned that, under certain restrictions on $X$, the classical Lagrange interpolation polynomial satisfies assumptions~\eqref{a1}, \eqref{a2}, and~\eqref{a3}.

In this section, we investigate assumptions~\eqref{a1}, \eqref{a3}, and~\eqref{a4} for more general quasi-interpolation operators
$$
Q_n^\vp f(x):=\frac{1}{2n+1}\sum_{k=0}^{2n}f(t_k)\vp_n(x-t_k),\quad t_k=\frac{2\pi k}{2n+1},
$$
where the kernel $\vp_n$ has the form
$$
\vp_n(x)=\sum_{k=-n}^n \vp\(\frac kn\)e^{ikx}
$$
and $\vp$ is an even function of bounded variation with $\supp\vp\subset [-1,1]$.

\begin{proposition}\label{prex1}
   Let $X$ be a Banach lattice such that
   \begin{equation}\label{zv1}
     \sup_{n\in \N,\, \|f\|_X\le 1}\|f*\vp_n\|_X<\infty,
   \end{equation}
   then the operator $Q_n^\vp$ satisfies~\eqref{a1}.

   In particular, \eqref{zv1} holds if $X$ satisfies Assumption A
and $\widehat{\vp}\in \mathcal{R}^*$.
\end{proposition}

\begin{proof}
By duality and H\"older's inequality, we obtain
  \begin{equation}\label{prex1e1}
  \begin{split}
         \|Q_n^\vp f\|_X&=\sup_{\|g\|_{X'}\le 1}\int_\T \frac{1}{2n+1}\sum_{k=0}^{2n}f(t_k)\vp_n(x-t_k)\overline{g(x)}d\mu(x)\\
    &=\sup_{\|g\|_{X'}\le 1}\int_\T \(\sum_{k=0}^{2n}f(t_k)\chi_{[t_k,t_{k+1})}(x)\)
    \(\sum_{\ell=0}^{2n}\vp_n*\overline{g}(t_\ell)\chi_{[t_\ell,t_{\ell+1})}(x)\)d\mu(x)\\
    &\le \sup_{\|g\|_{X'}\le 1} \bigg\|\sum_{k=0}^{2n}f(t_k)\chi_{[t_k,t_{k+1})}\bigg\|_X\bigg\|\sum_{\ell=0}^{2n}\vp_n*\overline{g}(t_\ell)\chi_{[t_\ell,t_{\ell+1})}\bigg\|_{X'}\\
    &\le \sup_{\|g\|_{X'}\le 1} \|\vp_n*\overline{g}\|_{X_n'}\|f\|_{X_n}.
  \end{split}
  \end{equation}
Applying Lemma~\ref{lemmz0} and, again, by duality arguments, we get
  \begin{equation}\label{prex1e2-}
  \begin{split}
        \|\vp_n*\overline{g}\|_{X_n'}\le C\|\vp_n*\overline{g}\|_{X'}\le C\|\overline{g}\|_{X'}\le C.
  \end{split}
  \end{equation}
Thus, combining~\eqref{prex1e1} and~\eqref{prex1e2-}, we prove that  $Q_n^\vp$ satisfies~\eqref{a1}.

We now show that if $X$ satisfies Assumption A and $\widehat{\vp}\in \mathcal{R}^*$, then
\begin{equation}\label{vpzv}
  \vp_n\in\mathcal{R}^*\quad\text{with}\quad\sup_n \|\vp_n^*\|_{L_1(\T)}<\infty.
\end{equation}
To this end, we use the following general result from~\cite[4.1.2]{TB}: if $\phi:\mathbb{R}\to\mathbb{C}$ is a function of bounded variation such that
$\lim_{|x|\to\infty} \phi(x)=0$, then, for all $\varepsilon>0$, we have
\begin{equation}\label{trigub}
  \sup_{0<|x|\le \pi}
\left|
\varepsilon \sum_{k=-\infty}^{\infty} \phi(\varepsilon k)e^{ikx}
- \int_{-\infty}^{\infty} \phi(u)e^{iux/\varepsilon}\,du
\right|
\le 2\varepsilon V(\phi),
\end{equation}
where $V(\phi)$ denotes the total variation of $\phi$ on $\mathbb{R}$. It is easy to see that if, in addition, $\phi$ has a compact support, then this inequality is valid for all $|x|\le\pi$.

Using~\eqref{trigub}, we obtain
\begin{equation*}
      |\vp_n(x)|\le |n\widehat{\vp}(nx)|+2V(\vp)\le |n\widehat{\vp}^*(nx)|+2V(\vp).
\end{equation*}
Since
\begin{equation*}
\|n\widehat{\vp}^*(n\cdot)\|_{L_1(\T)}\le \|\widehat{\vp}^*\|_{L_1(\R)}<\infty,
\end{equation*}
the function $\vp_n$ satisfies~\eqref{vpzv}.

Thus, applying  Lemma~\ref{lemS}, we get
  \begin{equation*}
  \begin{split}
        \|\vp_n*f\|_{X}\le C\|\vp_n^*\|_{L_1(\T)}\|f\|_{X}\le C\|f\|_{X},
  \end{split}
  \end{equation*}
which yields~\eqref{zv1}.
\end{proof}

The following remark provides simple conditions on $f\in L_1(\R)$ to ensure that $\widehat{f}\in \mathcal{R}^*$.

\begin{remark}\label{reml1}
Recall that for $f\in L_1(\R)$ and $r\ge 2$, $r\in \N$, the following inequality holds (see, e.g.,~\cite{Treb77}):
$$
|\widehat{f}(\xi)|\le C\w_r\(f,\frac1{|\xi|}\)_{L_1(\R)},
$$
where the constant $C$ is independent of $f$ and $\xi$.
In particular, if $\displaystyle\int_0^1\frac{\w_r(f,t)_{L_1(\R)}}{t^2}dt<\infty$  or, stronger,  $f\in W_1^2(\R)$, then $\widehat{f}\in \mathcal{R}^*$.  This follows from the fact that the modulus of smoothness is an increasing function and $\w_r(f,h)_{L_1(\R)}\le Ch^2\|f''\|_{L_1(\R)}$, see, e.g.,~\cite{KT20}.
\end{remark}

The next proposition provides conditions on $\vp$ under which assumptions~\eqref{a3} and~\eqref{a4} are satisfied for $Q_n^\vp$.
To formulate it, we denote by $\eta$ a function in $C^\infty(\R)\cap \mathcal{R}$ such that $\eta(\xi)=1$ for $|\xi|\le 1$ and $\eta(\xi)=0$ for $|\xi|\ge 2$.

\begin{proposition}\label{prex2}
   Let $X$ be a Banach lattice satisfying Assumption A.
\begin{itemize}
\item[(i)] Let $$w_1(\xi)=\frac{1-\varphi(\xi)}{\xi^s}\eta(\xi).$$
If $w_1$ is of bounded variation and $\widehat{w_1}\in\mathcal R^*$, then $Q_n^\varphi$ satisfies~\eqref{a3}.

\item[(ii)] Let $$w_2(\xi)=\frac{\xi^s\eta(\xi)}{1-\varphi(\xi)}.$$
If $w_2$ is of bounded variation and $\widehat{w_2}\in\mathcal R^*$, then $Q_n^\varphi$ satisfies~\eqref{a4}.
\end{itemize}
Here $w_1$ and $w_2$ are understood to be continuously extended at $\xi=0$.
\end{proposition}

\begin{proof}
  We prove only the first statement. By the quadrature formula for trigonometric polynomials, we have
   $$
   Q_n^\vp T=T*\vp_n,\quad T\in \mathcal{T}_n.
   $$
  This together with the principle of comparison for Fourier multipliers (see, e.g.,~\cite[Ch.~7]{TB} or \cite[Sec.~2.4]{V23}) implies that inequality~\eqref{a3} holds for $Q_n^\vp$  if and only if there exists a constant $C>0$ such that
  \begin{equation*}
    \|T*W_n\|_X\le C\|T\|_X,\quad T\in\mathcal{T}_n,
  \end{equation*}
  where
  $$
  W_n(x)=\sum_{k}w_1\(\frac kn\)e^{ikx}.
  $$
  Therefore, to prove the statement (i), it is enough to check that $W_n\in\mathcal{R}^*$ and $\sup_n \|W_n\|_{L_1(\T)}<\infty$. This can be done by repeating the second part of the proof of Proposition~\ref{prex1}.
\end{proof}

\begin{example}\label{ex1}
We again consider the Lagrange interpolation operators
$$
\mathcal{L}_n f(x)=\frac{1}{2n+1}\sum_{k=0}^{2n} f(t_k)D_n(x-t_k),
$$
where $D_n$ is the corresponding Dirichlet kernel.

We note that $\mathcal{L}_n=Q_n^\vp$ with $\vp_n=D_n$ and $\vp=\chi_{[-1,1]}$. Thus, applying Proportions~\ref{prex1} and~\ref{prex2}  together with Corollary~\ref{cor2t} and the equalities $S_n f=f*D_n$ and $\mathcal{L}_nf(t_k)=f(t_k)$, $k=0,\dots,2n$,
we obtain the following statement.

\begin{corollary}\label{cor2+}
Let $X$ be a Banach lattice satisfying Assumption~A, and let $r, s \in \N$ with $2r\ge s$. Suppose that  $\sup_n\|S_n\|_{X\to X}<\infty$. Then the following properties are equivalent for all $\a \in (0,s)$:
  \begin{itemize}
    \item[(i)]  $\|f-\mathcal{L}_n f\|_X=\mathcal{O}(n^{-\a})$,
    \item[(ii)] $\|(I-A_{\g/n})^r f\|_{X_n}+\|(I-\dot A_{\g/n})^s f\|_{X}=\mathcal{O}(n^{-\a})$.
  \end{itemize}
\end{corollary}

An alternative way to prove this result is to apply Corollary~\ref{cor2} together with Lemma~\ref{lemmz}.

\end{example}

\begin{example}\label{ex2}
Consider operators generated by the Bochner-Riesz and Fej\'er kernels:
$$
R_n(x):=\sum_{k=-n}^n \rho\(\frac kn\)e^{ikx},\quad \rho(\xi)=(1-\xi^2)_+^\a,\quad \a>0,
$$
$$
F_n(x):=\sum_{k=-n}^n \phi\(\frac kn\)e^{ikx},\quad \phi(\xi)=(1-|\xi|)_+.
$$
It is well known that there exists a constant $C>0$ such that
\begin{equation}\label{phirho}
  |\widehat{\rho}(x)|\le \frac{C}{(1+|x|)^{1+\a}}\quad\text{and}\quad
|\widehat{\phi}(x)|\le \frac{C}{(1+|x|)^2}\quad\text{for all }x\in \R.
\end{equation}
These estimates together with Proposition~\ref{prex1} imply that $Q_n^\rho$ and $Q_n^\phi$ satisfy assumption~\eqref{a1}.

We now show that the Bochner-Riesz sampling operator $Q_n^\rho$ satisfies assumptions~\eqref{a3} and~\eqref{a4}.

Denote
$$
w_1(\xi)=\frac{1-\rho(\xi)}{\xi^2}\eta(\xi),\qquad w_2(\xi)=\frac{\xi^2\eta(\xi)}{1-\rho(\xi)},
$$
$$
\eta_1(\xi)=\eta(4\xi),\qquad\eta_2(\xi)=\eta(\xi)-\eta_1(\xi),
$$
and
$$
g_1(\xi)=\frac{1-\rho(\xi)}{\xi^2}\eta_1(\xi),\qquad g_2(\xi)=\frac{\eta_2(\xi)}{\xi^2}.
$$
We have
$$
w_1(\xi)=g_1(\xi)+g_2(\xi)-\rho(\xi)g_2(\xi).
$$

Since $g_1$ and $g_2$ belong to $C^\infty(\R)$ and have compact supports, Remark~\ref{reml1} yields $\widehat{g_1}$, $\widehat{g_2}\in\mathcal{R}^*$. The first inequality in~\eqref{phirho} implies directly $\widehat{\rho}\in \mathcal{R}^*$. Furthermore, it follows from
$$
|\widehat{\rho g_2}|=|\widehat{\rho}*\widehat{g_2}|\le \widehat{\rho}^**|\widehat{g_2}|\in \mathcal{R}
$$
that $\widehat{\rho g_2}\in \mathcal{R}^*$. Hence $\widehat{w_1}\in \mathcal{R}^*$.

Similarly, one verifies that $\widehat{w_2}\in \mathcal{R}^*$ using the necessary technical details from~\cite[Proof of Theorem~13]{RS08}.
Then, applying Proposition~\ref{prex2}, we obtain that $Q_n^\rho$ satisfies assumptions~\eqref{a3} and~\eqref{a4}.

The same approach shows that the Fej\'er sampling operator $Q_n^\phi$ satisfies
\begin{equation*}
  \|T-Q_n^\phi T\|_X\le \frac {C_1}n\|\Delta^{1/2}T\|_X, \quad T\in \mathcal{T}_n,
\end{equation*}
and
\begin{equation*}
\frac{C_2}n\|\Delta^{1/2}T\|_X \le \|T-Q_n^\phi T\|_X, \quad T\in \mathcal{T}_n,
\end{equation*}
where $C_1$ and $C_2$ are positive constants independent of $T$ and $n$. Thus, if we additionally assume that the Hilbert transform\footnote{Recall that the periodic Hilbert transform is defined by
$
Hf(x) = \sum_{k\in\mathbb Z} (-i\,\mathrm{sgn}(k))\,\widehat f_k\,e^{ikx}$,
$\mathrm{sgn}(0)=0.
$
} is bounded in $X$, that is $\|\Delta^{1/2}T\|_X\asymp \|T'\|_X$ for any trigonometric polynomial $T$, then we obtain that $Q_n^\phi$ satisfies assumptions~\eqref{a3} and~\eqref{a4} with $s=1$.

Summarizing the above arguments, we get the following theorem, which provides the strong converse inequalities for  $Q_n^\phi$ and $Q_n^\rho$.

\begin{theorem}\label{br}
  Let $X$ be a Banach lattice satisfying Assumption A. Then, for all $f\in X$ and $r\in\N$, we have
  \begin{equation*}
    \|f-Q_n^\rho f\|_{X}+\|f-Q_n^\rho f\|_{X_n}\asymp \|(I-\dot A_{\g/n})^2 f\|_{X}+\|(I-A_{\g/n})^r f\|_{X_n}.
  \end{equation*}
If, additionally, the Hilbert transform $H$ is bounded on $X$, then
\begin{equation*}
    \|f-Q_n^\phi f\|_{X}+\|f-Q_n^\phi f\|_{X_n}\asymp \|(I-\dot A_{\g/n}) f\|_{X}+\|(I-A_{\g/n})^r f\|_{X_n},
  \end{equation*}
Here $\asymp$ denotes a two-sided inequality with positive constants independent of $f$ and $n$.
\end{theorem}

\end{example}

We note that Theorem~\ref{br} is new even in the classical case $X=L_p(\T)$, $1\le p<\infty$.

\section{Main results in the non-periodic case}
Analogues of the results from Section~\ref{main} remain valid for non-periodic functions $f$  defined on $\R$ and sampling operators
\begin{equation*}
\mathcal{G}_\s f(x):=\sum_{k\in \Z} f(x_{k,\s})\vp_{k,\s}(x),
\end{equation*}
where $\vp_{k,\s} \in \Phi_\s$, $k\in \Z$, and  $\mathcal{X}_\s=(x_{k,\s})_{k\in\Z}$ is a set of distinct points for which there exist positive constants $\g$ and $\g'$ satisfying
\begin{equation*}
  \frac{\g}{\s}\le x_{k+1,\s}-x_{k,\s}\le \frac{\g'}{\s},\quad k\in \Z.
\end{equation*}
For convenience, we set $x_k=x_{k,\s}$ and $\mathcal{G}_0:=0$. 

Similarly to the periodic case, we introduce the following semi-norm
$$
\|f\|_{X_\s(\R)}:=\bigg\|\sum_{k\in \Z}|f(x_k)|\chi_{[x_k,x_{k+1})}\bigg\|_{X(\R)}.
$$
We say that $f\in X_\s(\R)$ if $f\in X(\R)$ and $\|f\|_{X_\s(\R)}<\infty$. We again emphasize that the functions are considered pointwise and not identified almost everywhere.

Furthermore, we assume that, for all $f\in X_\s(\R)$, the series defining
$\mathcal{G}_\s f$ converges in the norm of $X$. To ensure such convergence for sampling operators with non-compactly supported functions $\vp_{k,\s}$,
we additionally assume that the Banach lattice
$X$ has an order continuous norm and the Schwartz space $\mathcal{S}$ is dense in $X(\R)$. Recall
$X$ is said to have an order continuous norm if $f_n \downarrow 0$ a.e. implies $\|f_n\|_X\to 0$ (see, e.g.,~\cite[Ch. 2, Sec. 2.4]{MN91}).

In this section, for convenience, we write $X=X(\R)$ and $X_\s=X_\s(\R)$.

As in the periodic case, to formulate the main results, we use the following assumptions on
$\mathcal{G}_\s$, involving a fixed parameter $s\in \N$:

\begin{equation}\label{a1r}\tag{$A_1'$}
  \|\mathcal{G}_\s f\|_X\le K_1\|f\|_{X_\s},\quad f\in X_\s,\quad \s\ge 1,
\end{equation}

\begin{equation}\label{a2r}\tag{$A_2'$}
  K_2\|f\|_{X_\s}\le \|\mathcal{G}_\s f\|_X,\quad f\in X_\s,\quad \s\ge 1,
\end{equation}

\begin{equation}\label{a3r}\tag{$A_3'$}
  \|g-\mathcal{G}_\s g\|_X\le K_3\s^{-s}\|g^{(s)}\|_X, \quad g\in \mathcal{B}_X^\s,\quad \s\ge 1,
\end{equation}

\begin{equation}\label{a4r}\tag{$A_4'$}
  K_4\s^{-s}\|g^{(s)}\|_X \le \|g-\mathcal{G}_\s g\|_X, \quad g\in \mathcal{B}_X^\s,\quad \s\ge 1.
\end{equation}

\smallskip

\noindent Here $K_1,K_2>0$ depend only on $X$ and $K_3,K_4>0$ depend only on $s$ and $X$.

\smallskip

The proofs of the main results in Section~\ref{main} extend to the spaces
$X(\R)$ and the operators
$\mathcal{G}_\s$ with only routine adjustments. Henceforth, we restrict ourselves to the formulation of the corresponding key statements and to a more detailed discussion of the examples, in order to justify the use of convolution operators and to verify the correctness of the definition of sampling operators given by series in $X(\R)$.

\subsection{Estimates in  terms of special moduli of smoothness}\label{secir}

\begin{theorem}\label{thstr}
Let ${X(\R)}$ be a Banach lattice satisfying Assumption~A, and let $r, s \in \N$ with $2r\ge s$. Suppose that $\mathcal{G}_\s$ satisfies assumptions~\eqref{a1r} and~\eqref{a3r}. Then, for all $f\in X_\s$, we have
\begin{equation*}
  \|f-\mathcal{G}_\s f\|_{X}\le C\(\|(I-\dot A_{\g/\s})^s f\|_{X}+\|(I-A_{\g/\s})^r f\|_{X_\s}\).
\end{equation*}
If, additionally, $\mathcal{G}_\s: X_\s \to \mathcal{B}_X^{\s}$, then
\begin{equation*}
  \|f-\mathcal{G}_\s f\|_{X_\s}\le C\(\|(I-\dot A_{\g/\s})^s f\|_{X}+\|(I-A_{\g/\s})^r f\|_{X_\s}\).
\end{equation*}
In both estimates, the constant $C$ is independent of $f$ and $\s$.
\end{theorem}

Similarly to Section~\ref{seci}, we obtain the following  inverse estimates for approximation by the sampling operators $\mathcal{G}_\s$.

\begin{theorem}\label{thinvr}
Let $X(\R)$ be a Banach lattice satisfying Assumption~A, and let $r, s \in \N$ with $2r\ge s$.
For all $f\in X_\s$, the following statements holds:

\begin{itemize}
  \item[(i)] If $\mathcal{G}_\s$ satisfies assumptions~\eqref{a1r}, \eqref{a2r}, and \eqref{a3r}, then
\begin{equation*}
  \|(I-\dot A_{\g/\s})^s f\|_{X}+\|(I-A_{\g/\s})^r f\|_{X_\s}\le  \frac{C}{\s^s}\sum_{k=0}^{\lfloor\s\rfloor} (k+1)^{s-1}\|f-\mathcal{G}_k f\|_{X}.
\end{equation*}
  \item[(ii)] If $\mathcal{G}_\s$ satisfies assumptions~\eqref{a1r} and \eqref{a3r}, then
\begin{equation*}
\begin{split}
    \|(I-&\dot A_{\g/\s})^s f\|_{X}+\|(I-A_{\g/\s})^r f\|_{X_\s}\\
    &\le C\(\|f-\mathcal{G}_{\s} f\|_{X}+\|f-\mathcal{G}_\s f\|_{X_\s}+\frac{1}{\s^s}\sum_{k=0}^{\lfloor\s\rfloor} (k+1)^{s-1}\|f-\mathcal{G}_{k} f\|_{X}\).
\end{split}
\end{equation*}
\end{itemize}
In both estimates, the constant $C$ is independent of $f$ and $\s$.
\end{theorem}

The application of Theorems~\ref{thstr} and~\ref{thinvr} yields the following corollaries.
We denote
$$
\mathcal{B}_X := \bigcup_{\s\ge 1} \mathcal{B}_X^\s.
$$

\begin{corollary}\label{cor1r}
Assume that the set $\mathcal{B}_X$ is dense in $X$. For  all $f\in \bigcap_{\s\ge 1} X_\s$, the following statements hold:

\begin{itemize}
  \item[(i)] If $\mathcal{G}_\s$ satisfies assumptions~\eqref{a1r}, \eqref{a2r}, and~\eqref{a3r}, then
  $$
    \|f-\mathcal{G}_\s f\|_{X}\to 0\quad\text{as } \s\to \infty
  $$
if and only if
\begin{equation*}
   \|(I-A_{\g/\s})^r f\|_{X_\s}\to 0\quad\text{as }\s\to \infty.
\end{equation*}

  \item[(ii)]  If $\mathcal{G}_\s: X_\s \to \mathcal{B}_X^{\s}$ satisfies assumptions~\eqref{a1r} and~\eqref{a3r}, then
   $$
   \|f-\mathcal{G}_\s f\|_{X}+\|f-\mathcal{G}_\s f\|_{X_\s}\to 0\quad\text{as } \s\to \infty
   $$
   if and only if
\begin{equation*}
   \|(I-A_{\g/\s})^r f\|_{X_\s}\to 0\quad\text{as } \s\to \infty.
\end{equation*}
\end{itemize}
\end{corollary}

\begin{corollary}\label{cor2r}
  Under the conditions of Theorem~\ref{thinvr}(i), the following properties are equivalent for all $\a \in (0,s)$:
  \begin{itemize}
    \item[(i)]  $\|f-\mathcal{G}_\s f\|_{X}=\mathcal{O}(\s^{-\a})$,
    \item[(ii)] $\|(I-A_{\g/\s})^r f\|_{X_\s}+\|(I-\dot A_{\g/\s})^s f\|_{X}=\mathcal{O}(\s^{-\a})$.
  \end{itemize}
\end{corollary}

\begin{corollary}\label{cor2r+}
  Under the conditions of Theorem~\ref{thinvr}(ii) with $\mathcal{G}_\s: X_\s \to \mathcal{B}_X^{\s}$, the following properties are equivalent for all  $\a \in (0,s)$:
  \begin{itemize}
    \item[(i)]  $\|f-\mathcal{G}_\s f\|_{X}+\|f-\mathcal{G}_\s f\|_{X_\s}=\mathcal{O}(\s^{-\a})$,
    \item[(ii)] $\|(I-\dot A_{\g/\s})^s f\|_{X}+\|(I-A_{\g/\s})^r f\|_{X_\s}=\mathcal{O}(\s^{-\a})$.
  \end{itemize}
\end{corollary}

The corresponding strong converse inequalities for sampling operators on $X(\R)$ have the following form.

\begin{theorem}\label{thCr}
  Let $X(\R)$ be a Banach lattice satisfying Assumption~A, and let $r, s \in \N$ with $2r\ge s$. Assume that $\mathcal{G}_\s: X_\s\to \mathcal{B}_X^{\s}$ satisfies assumptions~\eqref{a1r}, \eqref{a3r}, and~\eqref{a4r}.
  Then, for all $f\in X_\s$, we have
  \begin{equation*}
    \|f-\mathcal{G}_\s f\|_{X}+\|f-\mathcal{G}_\s f\|_{X_\s}\asymp \|(I-\dot A_{\g/\s})^s f\|_{X}+\|(I-A_{\g/\s})^r f\|_{X_\s},
  \end{equation*}
  where $\asymp$ denotes a two-sided inequality with positive constants independent of $f$ and $\s$.
\end{theorem}

\subsection{Estimates in terms of the errors of best and best one-sided approximations}

\begin{theorem}\label{thG1r}
Let $X(\R)$ be a Banach lattice satisfying Assumption~A. For all $f\in B\cap X$, the following statements hold:
\begin{itemize}
  \item[(i)]  If $\mathcal{G}_\s$ satisfies~\eqref{a1r} and~\eqref{a3r} for some $s\in \N$, then
\begin{equation*}
  \|f-\mathcal{G}_\s f\|_{X}\le C \(\widetilde{{E}}_\s(f)_{X}+\|(I-\dot A_{\g/\s})^s f\|_{X}\).
\end{equation*}
  \item[(ii)] If $\mathcal{G}_\s$ satisfies~\eqref{a1r} and  $\mathcal{G}_\s g = g$ for all $g\in \mathcal{B}_X^\s$, then
\begin{equation*}
  \|f-\mathcal{G}_\s f\|_{X}\le C \widetilde{{E}}_\s(f)_{X}.
\end{equation*}
\end{itemize}
In both estimates, the constant $C$ is independent of $f$ and $\s$.
\end{theorem}

\begin{corollary}\label{thG1corr}
Let $X(\R)$ be a Banach lattice satisfying Assumption~A, and let $s\in \N$. For all $f\in X^s$, the following statements hold:
\begin{itemize}
  \item[(i)]  If $\mathcal{G}_\s$ satisfies the assumptions of Theorem~\ref{thG1r}(i), then
\begin{equation*}
  \|f-\mathcal{G}_\s f\|_{X}\le \frac{C}{\s^s} \|f^{(s)}\|_{X}.
\end{equation*}

  \item[(ii)] If $\mathcal{G}_\s$ satisfies the assumptions of Theorem~\ref{thG1r}(ii), then
\begin{equation*}
  \|f-\mathcal{G}_\s f\|_{X}\le \frac{C}{\s^s} {E}_\s(f^{(s)})_{X}.
\end{equation*}
\end{itemize}
In both estimates, the constant $C$ is independent of $f$ and $\s$.
\end{corollary}

\begin{theorem}\label{thriGr}
  Let $X(\R)$ be a rearrangement-invariant Banach lattice. Assume that $\mathcal{G}_\s$ satisfies assumptions~\eqref{a1r} and~\eqref{a3r} for some $s\in \N$.
Then, for all $f\in X_\s$, we have
  \begin{equation*}
    \|f-\mathcal{G}_\s f\|_{X}\le C\(\w_s(f,1/\s)_{X}+\sum_{\nu=1}^\infty \|\d_{2^{-\nu}}\|_{X\to X} E_{2^{\nu-1}\s}(f)_{X}\),
  \end{equation*}
where the constant $C$ is independent of $f$ and $\s$. Moreover, if additionally, $\mathcal{G}_\s g=g$ for all $g\in \mathcal{B}_X^\s$, then the above estimate holds without the modulus $\w_s(f,1/\s)_{X}$ on the right-hand side.
\end{theorem}

\subsection{Examples} Similarly to Section~\ref{secext}, we discuss several classical examples of sampling operators on $\R$, for which the above results can be applied.

We denote
\begin{equation}\label{Qr}
      \mathcal{Q}_\s^\vp f(x):=\sum_{k\in\Z}f\(\s^{-1}k\)K^\vp\(\s x-k\),
\end{equation}
where $K^\vp (x):=\widehat{\vp}(x)$ and $\vp$ is an even, bounded function with a compact support.

In what follows, we set
$$
K_\s^\vp(x):=\s K^\vp (\s x).
$$

Since $K_\s^\vp \in \mathcal{B}_X^{\l\s}$ for some $\l>0$, it follows that, if $X$ satisfies Assumption~A, then $K_\s^\vp (\cdot-t)\in X$ for all $t\in \R$. Indeed, applying the Taylor formula and the Bernstein-type inequality from Lemma~\ref{lemm3}(i), we obtain
\begin{equation*}
  \begin{split}
      \|K_\s^\vp (\cdot-t)\|_X\le\sum_{\nu=0}^\infty \frac{\|(K_\s^\vp)^{(\nu)}\|_X}{\nu!}t^\nu\le e^{C\l\s t}\|K_\s^\vp\|_X<\infty.
   \end{split}
\end{equation*}
In particular, this shows that, for every $g\in X'$, the convolution $g*K_\s^\vp$ is well defined in the usual sense.

We now prove several propositions that can be used to verify assumption~\eqref{a1r}--\eqref{a4r} for the operators~$\mathcal{Q}_\s^\vp$.

\begin{proposition}\label{prex1r}
   Let $X(\R)$ be a Banach lattice satisfying Assumption~A such that
   \begin{equation}\label{zv1r}
     \sup_{\s\ge 1,\, \|g\|_{X'}\le 1}\|g*K_\s^\vp\|_{X'}<\infty,
   \end{equation}
   then, for all $f\in X_\s$, the series in~\eqref{Qr} converges in $X$, and the operator $\mathcal{Q}_\s^\vp$ satisfies~\eqref{a1r}.

   In particular, \eqref{zv1r} holds if $\widehat{\vp}\in \mathcal{R}^*$.
\end{proposition}

\begin{proof}
  For $m,n\in \Z_+$, $m\le n$, we set
  $$
  \mathcal{Q}_{\s, m,n}^\vp f(x):=\sum_{m\le |k|\le n} f\(t_k\)K^\vp\(\s(x-t_k)\),\quad t_k=\frac{k}{\s}.
  $$
  Using the same arguments as in the proof of Proposition~\ref{prex1} (see~\eqref{prex1e1} and~\eqref{prex1e2-}) together with Lemma~\ref{lemmz0r},  we obtain
  \begin{equation*}
  \begin{split}
    \|\mathcal{Q}_{\s, m,n}^\vp f\|_X &= \sup_{\|g\|_{X'}\le 1} \bigg|\int_{\R}\mathcal{Q}_{\s, m,n}^\vp f(x) \overline{g(x)}dx\bigg|\\
    &= \sup_{\|g\|_{X'}\le 1} \bigg|\sum_{m\le |k|\le n} f(t_k)\int_\R \overline{g(x)}K^\vp(\s(x-t_k))dx\bigg|\\
    &\le \sup_{\|g\|_{X'}\le 1} \s^{-1}\sum_{m\le |k|\le n} |f(t_k)(\overline{g}*K_\s^\vp) (t_k)|\\
    &\le \sup_{\|g\|_{X'}\le 1} \bigg\|\sum_{m\le |k|\le n} |f(t_k)|\chi_{[t_k,t_{k+1})}\bigg\|_X\\
    &\qquad\qquad\qquad\qquad\qquad\qquad\times\bigg\|\sum_{m\le |k|\le n} |(\overline{g}*K_\s^\vp) (t_k)|\chi_{[t_k,t_{k+1})}\bigg\|_{X'}
  \end{split}
  \end{equation*}
  and
  \begin{equation*}
  \begin{split}
    \bigg\|\sum_{m\le |k|\le n} |(\overline{g}*K_\s^\vp) (t_k)|\chi_{[t_k,t_{k+1})}\bigg\|_{X'}&\le \|\overline{g}*K_\s^\vp \|_{X_\s'}\\
    &\le C\|\overline{g}*K_\s^\vp\|_{X'}\le C\|\overline{g}\|_{X'} \le C,
  \end{split}
  \end{equation*}
  which yields
  \begin{equation}\label{prex1re2}
  \begin{split}
    \|\mathcal{Q}_{\s, m,n}^\vp f\|_X &\le C\bigg\|\sum_{|k|\ge m} |f(t_k)|\chi_{[t_k,t_{k+1})}\bigg\|_X.
  \end{split}
  \end{equation}
 Since $\sum_{k\in\mathbb Z}|f(t_k)|\chi_{[t_k,t_{k+1})} \in X$ and the corresponding
tails decrease pointwise to zero, the order continuity of the norm implies that
the right-hand side of \eqref{prex1re2} tends to zero as $m\to\infty$. This shows  that the series in~\eqref{Qr} converges in $X$, and  $\mathcal{Q}_\s^\vp$ satisfies~\eqref{a1r}.

  Finally, if $\widehat{\vp}\in \mathcal{R}^*$ , then applying  Lemma~\ref{lemS}, we obtain, for all $g\in X'$,
  \begin{equation}\label{prex1e2}
  \begin{split}
        \|g*K_\s^\vp\|_{X'}\le C\s \|\widehat{\vp}^*(\s\cdot)\|_{L_1(\R)}\|g\|_{X'}\le C\|g\|_{X'},
  \end{split}
  \end{equation}
which implies~\eqref{zv1r}.
\end{proof}

\begin{proposition}\label{prex1r+}
  Let $g\in \mathcal{B}_{X}^{(2\pi-a)\s}$ and $K_\s^\vp \in \mathcal{B}_{X'}^{a\s}$ for some $a\in (0,2\pi)$. Then
  \begin{equation}\label{po}
    Q_\s^\vp g=g*K_\s^\vp.
  \end{equation}
\end{proposition}

\begin{proof}
  Since $gK_\s^\vp(\cdot-t)\in L_1(\R)$ for all $t\in \R$, and since $\supp \mathcal{F}(gK_\s^\vp)\subset [-\pi\s,\pi\s]$, see, e.g.,~\cite[Theorem~6.37]{R73}, the Poisson summation formula yields~\eqref{po}.
\end{proof}

We now provide an analogue of Proposition~\ref{prex2}.
Recall that $\eta$ denotes a function in $C^\infty(\R)\cap \mathcal{R}$ such that $\eta(\xi)=1$ for $|\xi|\le 1$ and $\eta(\xi)=0$ for $|\xi|\ge 2$.

\begin{proposition}\label{prex2r}
   Let $X(\R)$ be a Banach lattice satisfying Assumption A, and let $\supp \vp\subset [-2\pi+1,2\pi-1]$.
\begin{itemize}
  \item[(i)] Let
   $$
   w_1(\xi)=\frac{1-\vp(\xi)}{\xi^s}\eta(\xi).
   $$
   If $\widehat{w_1}\in \mathcal{R}^*$, then $\mathcal{Q}_\s^\vp$ satisfies~\eqref{a3r}.
  \item[(ii)] Let
   $$
   w_2(\xi)=\frac{\xi^s \eta(\xi)}{1-\vp(\xi)}.
   $$
   If $\widehat{w_2}\in \mathcal{R}^*$, then $\mathcal{Q}_\s^\vp$ satisfies~\eqref{a4r}.
\end{itemize}
Here $w_1$ and $w_2$ are understood to be continuously extended at $\xi=0$.
\end{proposition}

\begin{proof}
  We prove only the first statement. Applying Proposition~\ref{prex1r+} for $g\in\mathcal{B}_X^\s$, we obtain
   \begin{equation*}
     g-\mathcal{Q}_\s^\vp g = g-g*K_\s^\vp = \mathcal{F}^{-1}\((\s^{-1}\xi)^s w_1(\s^{-1}\xi)\widehat{g}(\xi)\)
   \end{equation*}
   Using this identity together with the principle of comparison for Fourier multipliers (see, e.g.,~\cite[Ch.~7]{TB} or \cite[Sec.~2.4]{V23}), we conclude that  inequality~\eqref{a3r} holds for $\mathcal{Q}_\s^\vp$  if and only if there exists a constant $C>0$ such that
  \begin{equation}\label{ws}
    \sup_{\s\ge 1}\|g*\mathcal{W}_\s\|_X\le C\|g\|_X,\quad g\in\mathcal{B}_X^\s,
  \end{equation}
  where
  $$
  \mathcal{W}_\s(x) = \mathcal{F}^{-1} \(w_1\(\s^{-1} \xi\)\)(x).
  $$
  Applying Lemma~\ref{lemS}  by the same way as in~\eqref{prex1e2}, we prove~\eqref{ws}.
\end{proof}

\begin{example}\label{ex1r}
  Consider the Whittaker--Kotelnikov--Shannon sampling expansion defined by
$$
\mathcal{S}_\s f(x):=\sum_{k\in\Z}f\(\s^{-1}k\){\rm sinc}(\s x-k),
$$
where
$$
{\rm sinc}(x):=\frac{\sin \pi x}{\pi x}.
$$

Note that $\mathcal{S}_\s=\mathcal{Q}_\s^\vp$ with $\vp=\chi_{[-\pi\s,\pi\s]}$.

Denote
$$
\mathcal{I}_\sigma f(x) := \mathcal F^{-1}\big(\chi_{[-\pi\sigma,\pi\sigma]}(\xi)\,\widehat f(\xi)\big)(x),\quad f\in \mathcal{S}.
$$

\begin{corollary}\label{cor2rs}
  Let $X(\R)$ be a Banach lattice satisfying Assumption~A with ${\rm sinc}\in X$ and $\sup_{\sigma\ge1}\|\mathcal{I}_\sigma\|_{X\to X}<\infty$, and let $r, s \in \N$ with $2r\ge s$. Then the following properties are equivalent for all $f\in \bigcap_{\s\ge 1}X_\s(\R)$ and $\a \in (0,s)$:
  \begin{itemize}
    \item[(i)]  $\|f-\mathcal{S}_\s f\|_{X}=\mathcal{O}(\s^{-\a})$,
    \item[(ii)] $\|(I-\dot A_{\g/\s})^s f\|_{X}+\|(I-A_{\g/\s})^r f\|_{X_\s}=\mathcal{O}(\s^{-\a})$.
  \end{itemize}
\end{corollary}

\begin{proof}
By duality, we have $\sup_{\sigma\ge1}\|\mathcal{I}_\sigma\|_{X'\to X'}<\infty$. Moreover, $\mathcal{I}_\s f=f*{\rm sinc}_\s$ for all $f\in X'$. Together with Proposition~\ref{prex1r}, this implies that for every $f\in X_\s$ the series $\mathcal{S}_\s f$ converges in $X$, and  $\mathcal{S}_\s$ satisfies assumption~\eqref{a1r}.
Applying Proposition~\ref{prex2r}(i), we also obtain that $\mathcal{S}_\s$ satisfies assumption~\eqref{a3r} for all $s\in \N$.

Finally, applying Corollary~\ref{cor2r+} and using the identity  $\mathcal{S}_\s f(\s^{-1}k) = f(\s^{-1}k)$ for all $k\in \Z$, we complete the proof.
\end{proof}
\end{example}

\begin{example}\label{ex2r}
We now consider the Bochner-Riesz sampling operator $\mathcal{Q}_\s^\rho$, where $\rho(\xi)=(1-\xi^2)_+^\a$, $\a>0$, and
the Fej\'er operator $\mathcal{Q}_\s^\phi$, where $\phi(\xi)=(1-|\xi|)_+$.
By the same arguments as in Example~\ref{ex2}, and using Propositions~\ref{prex1r} and~\ref{prex2r} together with inequalities~\eqref{phirho}, we obtain the following strong converse inequalities for these operators.
\begin{theorem}\label{brr}
  Let $X(\R)$ be a Banach lattice satisfying Assumption A. Then, for all $f\in X_\s$ and $r\in \N$, we have
  \begin{equation*}
    \|f-\mathcal{Q}_\s^\rho f\|_{X}+\|f-\mathcal{Q}_\s^\rho f\|_{X_\s}\asymp \|(I-\dot A_{\g/\s})^2 f\|_{X}+\|(I-A_{\g/\s})^r f\|_{X_\s}.
  \end{equation*}
If, additionally, the Hilbert transform is bounded on $X$, then
\begin{equation*}
    \|f-\mathcal{Q}_\s^\phi f\|_{X}+\|f-\mathcal{Q}_\s^\phi f\|_{X_\s}\asymp \|(I-\dot A_{\g/\s}) f\|_{X}+\|(I-A_{\g/\s})^r f\|_{X_\s}.
  \end{equation*}
Here $\asymp$ denotes a two-sided inequality with positive constants independent of $f$ and $\s$.
\end{theorem}
\end{example}

\end{document}